\documentclass[a4paper]{article}

\usepackage{todonotes}
\usepackage{version} 
\usepackage[normalem]{ulem}
\usepackage{listings}   
\usepackage{xcolor}
\usepackage{amsfonts,amsthm,amsmath}
\usepackage{etoolbox}
\usepackage{authblk}
\usepackage{amsopn}
\usepackage{mathtools} 

\usepackage[backend=bibtex, style=numeric-comp, maxnames=700]{biblatex}
\bibliography{bib_johan_17_04}

\usepackage[USenglish]{babel}
\usepackage[babel]{csquotes}

\usepackage{microtype} 

\usepackage{subcaption}
\usepackage{amssymb}

\usepackage{tikz}
\usetikzlibrary{shapes,arrows}
\usetikzlibrary{spy}
\usepackage{pgfplots}
\pgfplotsset{compat=1.13}

\usepackage{booktabs}

\usepackage{environ}

\usepackage{hyperref}

\usepackage{bm}

\newcommand{\Real}{\ensuremath{\mathbb{R}}}
\newcommand{\RealExt}{\ensuremath{\overline{\Real}}}

\newcommand{\RecSpace}{\ensuremath{\mathcal{X}}}
\newcommand{\DataSpace}{\ensuremath{\mathcal{Y}}}

\newcommand{\PrimS}{\ensuremath{\mathcal{X}}}
\newcommand{\DualS}{\ensuremath{\mathcal{Y}}}
\newcommand{\GenericSpace}{\ensuremath{\mathcal{H}}}
\newcommand{\norel}{\ensuremath{\mathrel{\phantom{=}}}}
\newcommand{\weakto}{\ensuremath{\rightharpoonup}}

\DeclareMathOperator{\ForwardOp}{\ensuremath{T}}

\DeclareMathOperator{\LogLikelihood}{\ensuremath{\ell}}

\newcommand{\DataDiscrepancy}{D}
\newcommand{\Regularizer}{R}
\DeclareMathOperator{\Prox}{Prox}
\newcommand{\Id}{{\rm Id}}

\newcommand{\ee}{\ensuremath{\mathrm{e}}}
\newcommand{\dd}{\ensuremath{\mathrm{d}}}
\newcommand{\resolvent}[2]{\ensuremath{J^{#1}_{#2}}}
\newcommand{\MonotoneOpA}{\ensuremath{A}}
\newcommand{\MonotoneOpB}{\ensuremath{B}}
\newcommand{\GenericOp}{\ensuremath{A}}
\newcommand{\ExMatrix}{\ensuremath{A}} 
\newcommand{\SchemeMatrixA}{\ensuremath{\bm C}}
\newcommand{\SchemeMatrixB}{\ensuremath{\bm B}}
\newcommand{\SchemeMatrixC}{\ensuremath{\bm A}}
\newcommand{\SchemeMatrixD}{\ensuremath{\bm D}} 

\newcommand{\Lagr}[2]{\mathcal{L}\p{#1; #2}}

\newcommand{\signal}{\ensuremath{x}}
\newcommand{\truesignal}{\signal_{\text{true}}}

\newcommand{\data}{\ensuremath{b}}
\newcommand{\noise}{\delta\data}

\newcommand{\fstparam}{\ensuremath{c_{21}}}
\newcommand{\sndparam}{\ensuremath{a_{21}}}

\newcommand{\aoneone}{\ensuremath{c_{11}}}
\newcommand{\aonetwo}{\ensuremath{c_{12}}}
\newcommand{\atwotwo}{\ensuremath{c_{22}}}

\newcommand{\coneone}{\ensuremath{a_{11}}}
\newcommand{\conetwo}{\ensuremath{a_{12}}}
\newcommand{\ctwotwo}{\ensuremath{a_{22}}}

\newcommand{\fstparamP}{\ensuremath{c_{21}}}
\newcommand{\sndparamP}{\ensuremath{a_{21}}}

\newcommand{\Probdist}{\mathcal{P}}
\DeclareMathOperator{\Expect}{\mathbb{E}}

\DeclareMathOperator*{\argmin}{arg \, min}
\DeclareMathOperator{\diag}{{\rm diag}}

\newcommand{\Ordo}{\mathcal{O}}

\newcommand{\INNER}[2]{\ensuremath{\mathord{\left\langle #1, #2 \right\rangle}}}
\newcommand{\norm}[1]{\ensuremath{\mathord{\left\lVert #1 \right\rVert}}}

\newcommand{\StepLength}{\ensuremath{\sigma}}

\newcommand{\p}[1]{\ensuremath{\mathord{\left( #1 \right)}}}

\newcommand{\set}[1]{\ensuremath{\mathord{\left\lbrace #1 \right \rbrace}}}
\newcommand{\setcond}[2]{\ensuremath{\mathord{\left\lbrace #1 \; \middle\vert \;
      #2 \right \rbrace}}}
\newcommand{\inpr}{\INNER}
\newcommand{\defeq}{\coloneqq} 

\usepackage{enumerate}
\usepackage{mathrsfs}
\DeclareMathOperator{\graph}{graph}
\newcommand{\stochastic}[1]{\mathsf{#1}}
\newcommand{\stb}{\stochastic{b}}

\graphicspath{{figures/}
{figures/Training_2018_04_05to06/}
{figures/Training_2018_04_05to06/Deconv/}}

\newtheorem{theorem}{Theorem}[section]
\newtheorem{lemma}[theorem]{Lemma}
\newtheorem{corollary}[theorem]{Corollary}
\newtheorem{proposition}[theorem]{Proposition}
\theoremstyle{remark}
\newtheorem{example}[theorem]{Example}
\newtheorem{remark}[theorem]{Remark}

\theoremstyle{definition}
\newtheorem{algorithm}[theorem]{Algorithm}

\numberwithin{equation}{section}

\title{Data-driven nonsmooth optimization}
	
\author[1]{Sebastian Banert\thanks{equal contribution}}
\author[1]{Axel Ringh$^*$}
\author[1,2]{Jonas Adler}
\author[1]{\\Johan Karlsson}
\author[1]{Ozan \"{O}ktem}

\affil[1]{KTH Royal Institute of Technology, 100 44 Stockholm, Sweden.} \affil[2]{Elekta, Box 7593, 103 93 Stockholm, Sweden.
\newline
  Email: \{\texttt{banert}, \texttt{aringh}, \texttt{jonasadl}, \texttt{ozan}\}%
  \texttt{@kth.se}, \texttt{johan.karlsson@math.kth.se}
}

\begin{document}
\maketitle

\begin{abstract}
In this work, we consider methods for solving large-scale optimization problems
with a possibly nonsmooth objective function. 
The key idea is to first specify a class of optimization algorithms using a generic iterative scheme involving only linear operations and
applications of proximal operators. 
This scheme contains many modern primal-dual first-order solvers like the Douglas--Rachford and hybrid gradient methods as special cases.
Moreover, we show convergence to an optimal point for a new method which also belongs
to this class.
Next, we interpret the generic scheme as a neural network and use
unsupervised training to learn the best set of parameters for a specific
class of objective functions while imposing a fixed number of iterations.
In contrast to other approaches of \enquote{learning to optimize}, we present an approach which
learns parameters only in the set of convergent schemes.
As use cases, we consider optimization problems arising in tomographic
reconstruction and image deconvolution,
and in particular a family of
total variation regularization problems.
\end{abstract}

\section{Introduction}
Many problems in science and engineering can be formulated as convex
optimization problems which then need to be solved accurately and efficiently.
In this paper we focus on methods for solving such problems, namely of the form
\begin{equation}\label{eq:general_form}
 \min_{x \in \PrimS} \Bigl[ F\p{x} + \sum_{i = 1}^m G_i\p{L_i x} \Bigr].
\end{equation}
Here, $L_i \colon \PrimS \to \DualS_i$, $i = 1, \ldots, m$, are linear operators, where $\PrimS, \DualS_1, \ldots, \DualS_m$ are Hilbert spaces, and $F\colon
\PrimS \to \RealExt$ and $G_i\colon \DualS_i \to \RealExt$, $i = 1, \ldots, m$,
are proper, convex and lower semicontinuous functions.
This class of optimization problems appears for example in variational
regularization of inverse problems in imaging, such as X-ray computed
tomography (CT) \cite{natterer2001mathematical, natterer2001themathematics},
magnetic resonance imaging (MRI) \cite{brown2014magnetic}, and electron
tomography \cite{oktem2015mathematics}.

A key challenge is to handle the computational burden. 
In imaging, and especially so for three-dimensional imaging, the resulting
optimization problem is very high-dimensional even after clever digitization and
might involve more than one billion variables.
Moreover, many regularizers that are popular in imaging (see Section \ref{sec:simulation}), like those associated
with sparsity, result in a nonsmooth objective function.
These issues prevent usage of variational methods in time-critical applications, such as medical imaging in a clinical setting.
Modern methods which aim at overcoming these obstacles are typically based on the
proximal point algorithm \cite{rockafellar1976monotone} and operator splitting
techniques, see e.g., \cite{eckstein1992douglas, beck2009afast,
  chambolle2011first, boyd2011distributed, combettes2011proximal,
  combettes2012primal, he2012convergence, bot2013douglas, bot2015convergence,
  ko2017class, latafat2017asymmetric, bauschke2017convex} and references
therein.

The main objective of the paper is to offer a computationally tractable approach
for minimizing large-scale nondifferentiable, convex functions.
The key idea is to \enquote{learn} how to optimize from training data, resulting in an
iterative scheme that is optimal given a fixed number of steps, while its
convergence properties can be analyzed. We will make this precise in Section
\ref{sec:learning_opt_solver}.

Similar ideas have been proposed previously in \cite{gregor2010learning, li2016learning, andrychowicz2016learning}, but these approaches are either limited to specific classes of iterative schemes, like gradient-descent-like schemes \cite{li2016learning, andrychowicz2016learning} that are not applicable for nonsmooth optimization, or specialized to a specific class of regularizers as in \cite{gregor2010learning}, which limits the possible choices of regularizers and forward operators.
The approach taken here leverages upon these ideas and yields a general
framework for learning optimization algorithms that are applicable to solving
optimization problems of the type \eqref{eq:general_form}, inspired by the
proximal-type methods mentioned above.

A key feature is to present a general formulation that includes several existing
algorithms, among them the primal-dual hybrid gradient (PDHG) algorithm (also
called the Chambolle--Pock algorithm) \cite{chambolle2011first} and the
primal-dual Douglas--Rachford algorithm \cite{bot2013douglas} as a special case.
This means that the learning can be done in a space of schemes that includes these solvers as special cases.
Moreover, from the proposed parametrization we also derive a new optimization algorithm.
We demonstrate the performance of a solver based on this general formulation by
training in an unsupervised manner for two inverse problems: image
reconstruction in CT and deconvolution, both through
TV regularization.
In particular, we present a method to learn the parameters of a convergent
solver and demonstrate the improvement to the ad-hoc parameter choice.
Moreover, empirical results indicate that by using additional
parameters we can achieve improved performance.

The paper is organized as follows: In Section \ref{sec:background} we recall
elements of monotone operator theory and convex optimization, while setting up
the notation. In Section \ref{sec:new_solver}, we present and analyze a new
solver for monotone inclusions, which in Section \ref{subsec:convex_opt_solver}
is specialized to convex optimization problems of the form
\eqref{eq:general_form}. Section \ref{sec:learning_opt_solver} deals with the
notion of \enquote{learning} an optimization solver, and in Section
\ref{sec:simulation} we present numerical experiments for variational
regularization of inverse problems in imaging.

\section{Background}\label{sec:background}
Solving optimization problems of the type in \eqref{eq:general_form} are often
addressed using \emph{variable splitting} techniques, which work well if the
different terms are \enquote{simple} \cite{bauschke2017convex, combettes2011proximal,
  eckstein1989splitting}. To keep the discussion as general as possible and
since it does not add complexity to the proof of convergence, we will carry it
out for monotone inclusions instead of convex optimization problems.
The following subsections present necessary background material on monotone operators, convex optimization, and variable splitting.

\subsection{Fundamental notions}
Let $\GenericSpace$ be a real Hilbert space with the inner product $\inpr{\cdot}{\cdot}$. We
denote convergence in norm (or strong convergence) and weak convergence by $\to$
and $\weakto$, respectively.
A set-valued operator $S\colon \GenericSpace \rightrightarrows \GenericSpace$ is \emph{monotone} if
\[
  \inpr{z - z'}{w - w'} \geq 0 \quad \text{for all $z, z' \in \GenericSpace$, $w \in S\p{z}$, and $w' \in S\p{z'}$.}
\]
A monotone operator $S$ is called \emph{maximally monotone} if, in addition, the graph of $S$,
defined by $\graph(S) \defeq \setcond{\p{z, w} \in \GenericSpace \times \GenericSpace}{w \in S\p{z}}$,
is not properly contained in the graph of any other monotone operator, i.e.,
\[
  \p{z, w} \in \graph(S)
  \iff \text{$\inpr{z - z'}{w - w'} \geq 0$ for all $\p{z', w'} \in \graph(S)$.}
\]
A monotone operator is called \emph{strongly monotone} if there exists a $\mu >
0$ such that
\[
  \inpr{z - z'}{w - w'} \geq \mu \norm{z - z'}^2 \quad \text{for all $z, z' \in \GenericSpace$, $w \in S\p{z}$, and $w' \in S\p{z'}$.}
\]
Next, for any scalar $\sigma > 0$, the operator  $\resolvent{\sigma}{S} = \p{\Id + \sigma S}^{-1}$ is called the
\emph{resolvent operator} or \emph{proximal mapping}
\cite{rockafellar1976monotone}. It can be shown that $\resolvent{\sigma}{S}$ is a
single-valued operator $\GenericSpace \to \GenericSpace$ \cite[Proposition 23.8]{bauschke2017convex}.
Note that an efficient routine to evaluate $\resolvent{\sigma}{S}$ for all
$\sigma > 0$ also enables to evaluate the resolvent operator of $S^{-1}$ via
\begin{equation}\label{eq:Moreau_decomposition}
  \resolvent{\sigma}{S^{-1}}\p{z} = z - \sigma \resolvent{1/\sigma}{S}\p{z/\sigma}
\end{equation}
for $z \in \GenericSpace$ (see \cite[Proposition 23.20]{bauschke2017convex}).

A \emph{maximally monotone inclusion problem} is defined as the problem of finding a
point $z \in \GenericSpace$ such that $0 \in S\p{z}$, which we henceforth denote $z \in
\text{zer}(S)$.
In fact, it is easily seen that $z \in \text{zer}(S)$ is equivalent with $z$
being a fixed-point for the resolvent operator, i.e.,  $z = \resolvent{\sigma}{S}\p{z}$.

One reason for the interest in maximally monotone inclusion problems is that the
\emph{subdifferential} $\partial F$ of a proper, convex and lower
semicontinuous function $F \colon \GenericSpace \to \RealExt$ is a maximally monotone operator
\cite{moreau1965proximite}. Here, $\partial F: \GenericSpace \rightrightarrows \GenericSpace$ is defined to be
\[
  \partial F\p{x} \defeq \setcond{y \in \GenericSpace}{\forall \tilde x \in \GenericSpace\colon F\p{\tilde x}
  \geq F\p{x} + \inpr{y}{\tilde x - x}}
\]
if $F\p{x} \in \Real$ and $\partial F\p{x} = \emptyset$ if $F\p{x} \in \set{\pm
  \infty}$.
Moreover, the subdifferential at any minimizer of such a function
contains zero, so $F$ can be minimized by solving a
maximally monotone inclusion problem \cite[Theorem 16.3]{bauschke2017convex}. 
Note that we do not distinguish between local and global
minimizers, since any local minimizer of a convex function is global
\cite[Proposition 11.4]{bauschke2017convex}.

\begin{remark}
  A continuous linear operator $\GenericOp \colon \GenericSpace \to \GenericSpace$ of a Hilbert space $\GenericSpace$ into
  itself is maximally monotone if and only if it is accretive, i.e., if
  $\inpr{x}{\GenericOp x} \geq 0$ for all $x\in \GenericSpace$ \cite[Corollary
  20.28, see also Definition 2.23]{bauschke2017convex}, and it is the
  subdifferential $\partial f$ of a function $f\colon \GenericSpace \to \RealExt$ if and
  only if it is additionally symmetric \cite[Proposition
  2.51]{barbu2012convexity}. In particular the Volterra integral operator
  \cite[Example 4.4]{bauschke2010examples}
  \[
    \p{\GenericOp f}\p{t} = \int_0^t f\p{s} \,\dd s
  \]
  and its inverse are maximally
  monotone, but not the subdifferential of a proper, convex and lower
  semicontinuous function.
\end{remark}

For $F\colon \GenericSpace \to \RealExt$, the Fenchel dual (convex conjugate) function
$F^*\colon \GenericSpace \to \RealExt$ is defined by
\cite[Chapter~13]{bauschke2017convex}
\[
  F^*\p{y} \defeq \sup_{x\in \GenericSpace} \Bigl[ \inpr{x}{y} - F\p{x} \Bigr] \quad\text{for $y \in \GenericSpace$.}
\]
If $F$ is proper, convex and lower semicontinuous, then $\partial F^* =
\p{\partial F}^{-1}$ \cite[Corollary 16.30]{bauschke2017convex}.

The \emph{proximal point algorithm} is a fixed-point iterative scheme for solving the maximally monotone inclusion problem.
It is given by repeatedly applying the resolvent operator:
\[
  z^{k+1} = \resolvent{\sigma}{S} \p{z^k}.
\]
It can now be shown that if $\text{zer}(S) \neq \emptyset$ then $z^k$
converges weakly to a point $z^\infty \in \text{zer}(S)$
\cite{rockafellar1976monotone} for all starting points $z^0 \in \GenericSpace$.
The special case when $S \defeq \partial F$, i.e., the case of the resolvent of a
subdifferential of $F$, is called the \emph{proximal operator}. One can express the proximal as \cite{moreau1965proximite}
\begin{equation}\label{eq:proximal}
  \resolvent{\sigma}{\partial F} \p{x} = \Prox^{\sigma}_{F} \p{x} =
  \argmin_{x' \in \GenericSpace} \set{F\p{x'} + \frac{1}{2\sigma} \norm{x' - x }^2}.
\end{equation}
To see this, we simply note that if $x'$ is a minimizing argument then
\[
  0 \in \partial F\p{x'} + \frac{1}{\sigma} \p{x' - x} \quad \iff
  \quad x' = \resolvent{\sigma}{\partial F} \p{x}.
\]
It is thus interesting to note that the fixed-point iteration
\[
x^{k+1} = \Prox^{\sigma}_{F} \p{x^k} = \argmin_{x' \in \GenericSpace} \set{F\p{x'} +
  \frac{1}{2\sigma} \norm{x' - x^{k}}^2}
\]
generates a sequence $\p{x^k}$ that converges weakly to a minimizer of $F$.
In this setting, the parameter $\sigma$ can be interpreted as a step length.
This can give rise to methods for solving the optimization problems if
the proximal operator can be efficiently computed, e.g., through a closed-form
expression.
Note that \eqref{eq:Moreau_decomposition} gives a method to obtain the proximal
points of $F^*$ from those of $F$, namely
\[
  \Prox^\tau_{F^*} \p{x} = x - \tau \Prox^{1/\tau}_{F} \p{x/\tau} \qquad
  \text{for all } \tau > 0.
\]

Sometimes the resolvent of the maximally monotone operator $S$ is not easy to
evaluate, but $S$ is of the form $S = \MonotoneOpA + \MonotoneOpB$ where $\MonotoneOpA$ and $\MonotoneOpB$ are maximally
monotone and the resolvents  of $\MonotoneOpA$ and $\MonotoneOpB$ can be evaluated efficiently.
One may then consider approximating $\resolvent{\sigma}{\MonotoneOpA + \MonotoneOpB}$ with $\resolvent{\sigma}{\MonotoneOpA}$ and $\resolvent{\sigma}{\MonotoneOpB}$ (splitting)
\cite{eckstein1989splitting}.
An example when this arises is in convex minimization of an objective that is a sum of two (or more) functions $F
+ G$, like in \eqref{eq:general_form}.
In these cases it is often not possible to compute a closed-form expression for the proximal operator $\Prox^{\sigma}_{F + G}$.
Such problems can be addressed using operator splitting techniques that allow for solving the
problem by only evaluating $\Prox^{\sigma}_{F}$ and $\Prox^{\sigma}_{G}$
\cite{combettes2011proximal}.%
\footnote{In optimization, this operator splitting is sometimes referred to as
  \emph{variable splitting}. The reason for this can be understood by comparing equations
  \eqref{eq:primal_prob_cp} and \eqref{eq:opt_cond} below.}

\subsection{Convex optimization}\label{sec:OptBackground}
Next, we will consider duality and optimality conditions for the problem \eqref{eq:general_form}.
To simplify the notation, we consider the case $m=1$ in \eqref{eq:general_form}, i.e., let $\PrimS$ and $\DualS$ be two Hilbert spaces and
consider the model problem
\begin{equation}\label{eq:primal_prob_cp}
  \min_{x\in \PrimS} \Bigl[ F\p{x} + G\p{Lx} \Bigr],
\end{equation}
where $L\colon \PrimS \to \DualS$ is a continuous linear operator and $F\colon
\PrimS \to \RealExt$ and $G: \DualS \to \RealExt$ are proper, convex and lower
semicontinuous functions. 
Note that \eqref{eq:general_form} is recovered by setting 
\begin{equation}\label{eq:product_space_reduction}
G\p{y} \defeq \sum_{i = 1}^m G_i\p{y_i} 
  \quad\text{for $y = \p{y_1, \ldots, y_m} \in \DualS \defeq \DualS_1 \times \ldots \times \DualS_m$}
\end{equation}
and $Lx \defeq \p{L_1x, \ldots, L_m x}$ for $x\in \PrimS$ in \eqref{eq:primal_prob_cp}.

The dual formulation of the primal problem \eqref{eq:primal_prob_cp} is
\begin{equation}\label{eq:dual_prob_cp}
  \max_{y \in \DualS} \Bigl[ -F^*\p{L^*y} - G^*\p{-y} \Bigr].
\end{equation}
Under suitable conditions the two optimization problems
\eqref{eq:primal_prob_cp} and \eqref{eq:dual_prob_cp} have the same optimal
value \cite[Chapter~15.3]{bauschke2017convex}.
Also note that, since both $F$ and $G$ are proper, convex and lower
semicontinuous functions, $F^{**} =
F$ and $G^{**} = G$ by the Fenchel--Moreau theorem
\cite[Theorem~13.37]{bauschke2017convex}.
Hence, the following primal-dual formulation
\begin{equation}\label{eq:saddle_prob_cp}
\min_{x \in \PrimS} \, \max_{y \in \DualS} \; \Lagr{x}{y} \quad\text{with}\quad \Lagr{x}{y} \defeq \inpr{Lx}{y} + F\p{x} - G^*\p{y}
\end{equation}
(the mapping $\Lagr{\cdot}{\cdot}$ is called the \emph{Lagrangian}) is equivalent to the primal problem.%
\footnote{To see this, note that $\max_{y\in \DualS} \bigl[ \inpr{Lx}{y} - G^*\p{y} \bigr]=
  G^{**}\p{Lx} = G\p{Lx}$.}
In fact, under suitable assumptions it can be shown that if $\p{\bar{x}, \bar{y}}$ is a saddle point
to \eqref{eq:saddle_prob_cp}, then $\bar{x}$ is a
solution to the primal problem \eqref{eq:primal_prob_cp} and $\bar{y}$ is a
solution to the dual problem \eqref{eq:dual_prob_cp} \cite[Proposition 19.20]{bauschke2017convex}.

A necessary optimality condition for the primal-dual formulation
\eqref{eq:saddle_prob_cp} is that the corresponding point $\p{\bar x, \bar y} \in
\PrimS \times \DualS$ be stationarity with
respect to both variables, i.e., that
\begin{equation}\label{eq:primal_dual_solutions}
  L\bar x  \in \partial G^*\p{\bar y} \quad\text{and}\quad \mathord{-}L^*\bar y  \in \partial F\p{\bar x}.
\end{equation}
For later use we note that the first of these conditions can be reformulated as
\begin{align*}
  L\bar x \in \partial G^*\p{\bar y}
  & \quad \iff \quad
    \bar y + \sigma L \bar x \in \bar y + \sigma \partial G^*\p{\bar y}
    = \p{I + \sigma \partial G^*}\p{\bar y} \\
  & \quad \iff \quad
    \bar y = \resolvent{\sigma}{\partial G^*}\p{\bar y + \sigma L \bar x}
    = \Prox^\sigma_{G^*}\p{\bar y + \sigma L \bar x},
\end{align*}
and the second as
\begin{align*}
  -L^* \bar y \in \partial F\p{\bar x}
  & \quad \iff \quad
    \bar x - \tau L^* \bar y \in \bar x + \tau \partial F\p{\bar x}
    = \p{I + \tau \partial F}\p{\bar x} \\
  & \quad \iff \quad
    \bar x = \resolvent{\tau}{\partial F}\p{\bar x - \tau L^* \bar y}
    = \Prox^\tau_{F}\p{\bar x - \tau L^* \bar y}.
\end{align*}
Therefore, an equivalent condition to \eqref{eq:primal_dual_solutions} is
\begin{equation}\label{eq:opt_cond}
\bar y  = \Prox^\sigma_{G^*}\p{\bar y + \sigma L \bar x} 
\quad\text{and}\quad 
\bar x  = \Prox^\tau_{F}\p{\bar x - \tau L^* \bar y}. 
\end{equation}

\subsection{Two splitting algorithms}\label{subsec:few_algorithms}
As mentioned before, there are many different splitting methods available to solve
problems of the form \eqref{eq:general_form}.
For ease of reference, we here mention two popular choices.
The first one, given in \eqref{eq:alg_CP}, is PDHG \cite{chambolle2011first}
\begin{equation}\label{eq:alg_CP}
\begin{aligned}
  & y_{n+1} = \Prox^{\sigma}_{G^*} \p{y_n + \sigma L  v_n}, \\
  & x_{n+1} = \Prox^{\tau}_{F} \p{x_n - \tau L^* y_{n+1}}, \\
  & v_{n+1} = x_{n+1} + \theta \p{x_{n+1} - x_n}.
\end{aligned}
\end{equation}
The second one is the Douglas--Rachford type primal-dual algorithm
\cite{bot2013douglas}, presented in \eqref{eq:alg_DR}
\begin{equation}\label{eq:alg_DR}
\begin{aligned}
  p_n & = \Prox^\tau_F \p{x_n - \tau L^*y_{n}}, \\
  x_{n+1} & = x_n + \lambda_n \p{p_n - x_n}, \\
  q_n & = \Prox^{\sigma}_{G^*} \p{y_{n} + \sigma L\p{2p_n
              - x_n}}, \\
  y_{n+1} & = y_{n} + \lambda_n \p{q_n - y_{n}}.
\end{aligned}
\end{equation}

\section{A new family of optimization solvers}\label{sec:new_solver}
In this section we introduce a new family of optimization algorithms and prove
convergence for a subfamily.
For ease of notation we will consider the simplified optimization problem
\eqref{eq:primal_prob_cp}, but results easily extend to the general case
\eqref{eq:general_form}.

To this end, consider the two algorithms \eqref{eq:alg_CP} and
\eqref{eq:alg_DR}.
Note that they can both be written as
\begin{subequations}\label{eq:2x2_scheme_components}
  \begin{align}
    q_n &= \Prox^\sigma_{G^*} \p{b_{12}y_n
          + b_{11} L \p{\aoneone p_{n-1} + \aonetwo x_{n-1} }},\\
    y_{n+1} &= \sndparamP q_n + \ctwotwo y_n, \\
    p_n &= \Prox^\tau_{F} \p{d_{12} x_n
          + d_{11} L^* \p{\coneone q_n + \conetwo y_n}}, \\
    x_{n+1} &= \fstparamP p_n + \atwotwo x_n,
  \end{align}
\end{subequations}
for suitable values of the coefficients.
More precisely, the PDHG algorithm \eqref{eq:alg_CP} is obtained by setting
\begin{align*}
  \coneone &= 1           & \conetwo &= 0
  & \sndparamP &= 1  & \ctwotwo &= 0 & b_{11} &= \sigma  & b_{12} &= 1 \\
   \aoneone &= 1 + \theta  & \aonetwo &= -\theta
  & \fstparamP &= 1  & \atwotwo &= 0  & d_{11} &= -\tau   & d_{12} &= 1
\intertext{and the Douglas-Rachford algorithm \eqref{eq:alg_DR} by setting}
  \coneone &= \lambda_n  & \conetwo &= 1-\lambda_n  & \sndparamP &= \lambda_n
  & \ctwotwo &= 1-\lambda_n  & b_{11} &= \sigma  & b_{12} &= 1 \\
  \aoneone &= 2          & \aonetwo &= -1           & \fstparamP &= \lambda_n
  & \atwotwo &= 1-\lambda_n  & d_{11} &= -\tau    & d_{12} &= 1.
\end{align*}

We now go on to analyze the scheme \eqref{eq:2x2_scheme_components}. To state
our results as generally as possible, we formulate them for a monotone
inclusion problem that in particular specializes to the optimality conditions in
\eqref{eq:primal_dual_solutions} when the operators are subdifferentials.
The monotone inclusion problem we seek to solve reads as follows:
Let $\PrimS$ and $\DualS$ be two (not necessarily finite-dimensional) Hilbert
spaces, and let $L\colon \PrimS \to \DualS$ be a continuous linear operator.
Let $\MonotoneOpA: \PrimS \rightrightarrows \PrimS$ and $\MonotoneOpB: \DualS \rightrightarrows
\DualS$ be maximally monotone operators. Find a pair $\p{\bar x, \bar y} \in
\PrimS \times \DualS$ such that
\begin{equation}\label{eq:inclusion_problem}
  L\bar x \in \MonotoneOpB^{-1} \bar y 
  \quad\text{and}\quad
  \mathord{-}L^* \bar y \in \MonotoneOpA \bar x.
\end{equation}
In this setting, the scheme \eqref{eq:2x2_scheme_components} generalizes to
\begin{subequations}\label{eq:2x2_scheme}
  \begin{align}
    q_n &= \resolvent{\sigma}{\MonotoneOpB^{-1}}\p{b_{12}y_n + b_{11} L \p{\aoneone p_{n-1} + \aonetwo x_{n-1} }}, \label{eq:2x2_scheme_q} \\
    y_{n+1} &= \sndparamP q_n + \ctwotwo y_n, \label{eq:2x2_scheme_y} \\
    p_n &= \resolvent{\tau}{\MonotoneOpA} \p{d_{12} x_n + d_{11} L^* \p{\coneone q_n + \conetwo y_n}}, \label{eq:2x2_scheme_p} \\
    x_{n+1} &= \fstparamP p_n + \atwotwo x_n. \label{eq:2x2_scheme_x}
  \end{align}
\end{subequations}
We first note that if $\sndparamP = 0$ or $\fstparamP = 0$ the update for either
$y_{n+1}$ or $x_{n+1}$ becomes trivial, and the algorithm will not be globally
convergent to a point fulfilling \eqref{eq:inclusion_problem} in general.
Henceforth we will therefore assume that $\sndparamP$ and $\fstparamP$ are not equal
to $0$, unless the opposite is explicitly stated.

\subsection{Fixed-point analysis}\label{subsect:fixed_point_2x2}
In this section, we give necessary and sufficient conditions for the solution set of
\eqref{eq:inclusion_problem} and the fixed point set of \eqref{eq:2x2_scheme} to
coincide for any choice of $\MonotoneOpA$, $\MonotoneOpB$, and $L$.
To this end, let $\p{\bar q, \bar y, \bar p, \bar x} \in \DualS \times \DualS \times \PrimS \times \PrimS$ be a
fixed point of the iterative scheme \eqref{eq:2x2_scheme} and note that
\eqref{eq:2x2_scheme_y} and \eqref{eq:2x2_scheme_x} gives
\[
  \bar q = \frac{1 - \ctwotwo}{\sndparamP} \bar y
  \quad\text{and}\quad
  \bar p = \frac{1 - \atwotwo}{\fstparamP} \bar x.
\]
Using this, we further get that
\begin{align*}
  \frac{1 - \ctwotwo}{\sndparamP} \bar y
  &= \resolvent{\sigma}{\MonotoneOpB^{-1}} \p{b_{12}\bar y
    + b_{11} L \p{\aoneone \frac{1 - \atwotwo}{\fstparamP} \bar x + \aonetwo \bar x }} \\
  \frac{1 - \atwotwo}{\fstparamP} \bar x
  &= \resolvent{\tau}{\MonotoneOpA} \p{d_{12} \bar x
    + d_{11} L^* \p{\coneone \frac{1 - \ctwotwo}{\sndparamP} \bar y + \conetwo \bar y}}
\end{align*}
The conditions in \eqref{eq:inclusion_problem} can now be re-phrased as 
\[
  \bar y = \resolvent{\sigma}{\MonotoneOpB^{-1}} \p{\bar y + \sigma L\bar x}
  \quad\text{and}\quad
  \bar x = \resolvent{\tau}{\MonotoneOpA} \p{\bar x - \tau L^* \bar y},
\]
and combining the above two equations yields
\begin{equation}\label{eq:two_variables_fp_conditions}
\begin{aligned}
  \sndparamP + \ctwotwo &= 1,
  & b_{12} &= 1,
  & b_{11}\p{\aoneone + \aonetwo} &= \sigma, \\
  \fstparamP + \atwotwo &= 1,
  & d_{12} &= 1,
  & d_{11}\p{\coneone + \conetwo} &= -\tau. 
\end{aligned}
\end{equation}
The conditions in \eqref{eq:two_variables_fp_conditions} are necessary and sufficient,
however, due to the linearity of $L$, the algorithm does not change if we
agree to the normalization
\begin{align*}
  b_{11} &= \sigma,
  & \aoneone + \aonetwo &= 1, \\
  d_{11} &= -\tau,
  & \coneone + \conetwo &= 1.
\end{align*}
If we fix all these conditions, the iteration \eqref{eq:2x2_scheme}
takes the form
\begin{subequations}\label{eq:2x2_fixed}
  \begin{align}
    q_n &= \resolvent{\sigma}{\MonotoneOpB^{-1}} \p{y_n + \sigma L \p{x_{n-1} + \aoneone \p{p_{n-1} - x_{n-1}}}}, \label{eq:2x2_fixed_q} \\
    y_{n+1} &= y_n + \sndparamP \p{q_n - y_n}, \label{eq:2x2_fixed_y} \\
    p_n &= \resolvent{\tau}{\MonotoneOpA} (x_n - \tau L^* \p{y_n + \coneone \p{q_n - y_n}}), \label{eq:2x2_fixed_p} \\
    x_{n+1} &= x_n + \fstparamP \p{p_n -  x_n}. \label{eq:2x2_fixed_x}
  \end{align}
\end{subequations}

\subsection{Convergence analysis}
The following theorem gives sufficient conditions for the weak convergence of the sequence $\p{x_n, y_n}$ generated by \eqref{eq:2x2_fixed} to a point that satisfies \eqref{eq:inclusion_problem}, i.e., a point that solves the monotone inclusion problem.
\begin{theorem}\label{thm:convergence}
  Assume that there is a point that satisfies \eqref{eq:inclusion_problem},
  i.e., the monotone inclusion problem has a solution. Moreover, let 
  \begin{equation}\label{eq:a11andeq:c11}
     \coneone = \sndparamP \quad\text{and}\quad \aoneone = 1 + \frac{\fstparamP}{\sndparamP}.
  \end{equation}
  Assume furthermore that $0 < \sndparamP < 2$, $0 < \fstparamP < 2$ and 
  \begin{equation}\label{eq:sigma_tau_bound}
    \sigma \tau \norm{L}^2 < \frac{\sndparamP^2 \p{2 - \sndparamP} \p{2 - \fstparamP}}{\p{\sndparamP + \fstparamP - \sndparamP \fstparamP}^2} 
    \quad\text{with $\sigma, \tau > 0$.}
  \end{equation}
  Finally, let $\p{q_n, y_n, p_n, x_n}$ be the sequence generated by scheme
  \eqref{eq:2x2_fixed}. Then the following holds:
  \begin{enumerate}[(a)]
    \item \label{item:thm:convergence:primal_summability} 
      $\displaystyle{\sum_{n \geq 0}} \norm{x_n - p_n}^2 < +\infty$ and $\displaystyle{\sum_{n \geq 0}} \norm{x_n - x_{n+1}}^2 < +\infty$.
    \item \label{item:thm:convergence:dual_summability} 
      $\displaystyle{\sum_{n \geq 0}} \norm{y_n - q_n}^2 < +\infty$ and $\displaystyle{\sum_{n \geq 0}} \norm{y_n - y_{n+1}}^2 < +\infty$.
    \item \label{item:thm:convergence:weak_convergence} The sequence $\p{x_n,
        y_n}_n$ converges weakly to a point that satisfies \eqref{eq:inclusion_problem}.
    \item \label{item:thm:convergence:strong_convergence} 
      If $\MonotoneOpA$ is strongly monotone, then there is a unique $\bar x \in \PrimS$
      such that all solutions of \eqref{eq:inclusion_problem} are of the form
      $\p{\bar x, y}$ with some $y\in \DualS$. Moreover, $\displaystyle{\sum_{n=1}^\infty}
      \norm{p_n - \bar x}^2 < +\infty$, in particular $p_n \to \bar x$ strongly.
      \\
      If $\MonotoneOpB^{-1}$ is strongly monotone, then there is a unique $\bar y \in \DualS$
      such that all solutions of \eqref{eq:inclusion_problem} are of the form
      $\p{x, \bar y}$ with some $x \in \PrimS$. Moreover,
      $\displaystyle{\sum_{n=1}^\infty} \norm{q_{n+1} - \bar y}^2 < +\infty$, in
      particular $q_n \to \bar y$ strongly.
  \end{enumerate}
\end{theorem}
By rewriting with \eqref{eq:a11andeq:c11}, the iteration
\eqref{eq:2x2_fixed} takes the following form:
\begin{algorithm}\label{alg:mon_op}
  Choose parameters $\sigma > 0$, $\tau > 0$ and $\sndparam \in \Real$, $\fstparam \in
  \Real$ and starting points $x_0 \in \PrimS$, $x_1 \in \PrimS$, $p_0 \in
  \PrimS$, $y_1 \in \DualS$. For all $n = 1, 2, \ldots$, calculate
  \begin{subequations}\label{eq:stability}
    \begin{align}
      q_n &= \resolvent{\sigma}{\MonotoneOpB^{-1}} \p{y_n + \sigma L \p{p_{n-1} + \frac{\fstparam}{\sndparam} \p{p_{n-1} - x_{n-1}}}}, \label{eq:stability_q}\\
      y_{n+1} &= y_n + \sndparam \p{q_n - y_n}, \label{eq:stability_y} \\
      p_n &= \resolvent{\tau}{\MonotoneOpA} \p{x_n - \tau L^* y_{n+1}}, \label{eq:stability_p} \\
      x_{n+1} &= x_n + \fstparam \p{p_n -  x_n}. \label{eq:stability_x}
    \end{align}
  \end{subequations}
  Then, $x_n \weakto \bar x$, $p_n \weakto \bar x$, $y_n \weakto \bar y$, and $q_n
  \weakto \bar y$, where $\p{\bar x, \bar y}$ is a solution of
  \eqref{eq:inclusion_problem}, provided that $0 < \sndparam < 2$, $0 < \fstparam < 2$
  and \eqref{eq:sigma_tau_bound} are satisfied.
\end{algorithm}
The remainder of the convergence analysis will therefore refer to scheme
\eqref{eq:stability}.
The proof of Theorem \ref{thm:convergence} rests upon a number of technical results and is given in Section~\ref{subsec:ProofOfThm}. 
An immediate corollary is the convergence of the primal-dual
Douglas--Rachford method with constant relaxation \cite{bot2013douglas}.
\begin{corollary}
  Let $\sigma \tau \norm{L} < 1$ and $0 < \lambda < 2$. Then, for the iteration
  \begin{align*}
    q_n &= \resolvent{\sigma}{\MonotoneOpB^{-1}} \p{y_n + \sigma L\p{2p_{n-1} - x_{n-1}}}, \\
    y_{n+1} &= y_n + \lambda\p{q_n - y_n}, \\
    p_n &= \resolvent{\tau}{\MonotoneOpA} \p{x_n - \tau L^* y_{n+1}}, \\
    x_{n+1} &= x_n + \lambda \p{p_n - x_n},
  \end{align*}
  the sequence $\p{x_n, y_n}_n$ converges weakly to a point that satisfies \eqref{eq:inclusion_problem}.
\end{corollary}
\begin{proof}
  Set $\sndparam = \fstparam = \lambda$ in Theorem \ref{thm:convergence} and observe
  that \eqref{eq:sigma_tau_bound} reduces to $\sigma \tau \norm{L}^2 < 1$.
\end{proof}

\subsubsection{Proof of Theorem~\ref{thm:convergence}}\label{subsec:ProofOfThm}
For the proof, we define notions of distance $Q_1$ and $Q_2$ on the space
$\PrimS \times \DualS$ of pairs of primal and dual variables
(Lemma~\ref{lem:quadratic_forms}). Next, we show that the distance (in terms of
$Q_1$) between the iterates and the set of solutions of
\eqref{eq:inclusion_problem} decreases (Proposition~\ref{lem:stability}). This
property is also known as \emph{Fej\'er monotonicity} \cite[Chapter
5]{bauschke2017convex}. Proposition~\ref{lem:stability_strongly_monotone}
improves the statement of Proposition~\ref{lem:stability} for strongly monotone
operators. The proof of Theorem \ref{thm:convergence} is completed by showing
that any weak sequential cluster point of the iteration sequence is a solution
to \eqref{eq:inclusion_problem}.

We start with some simple inequalities between real numbers. In particular,
Lemma~\ref{lem:estimation}~\eqref{item:lem:estimation:1}
shows that we do not divide by zero in \eqref{eq:sigma_tau_bound}.

\begin{lemma}\label{lem:estimation}
  Let $0 < \sndparam < 2$ and $0 < \fstparam < 2$. Then
  \begin{enumerate}[(a)]
  \item \label{item:lem:estimation:1} $\sndparam + \fstparam > \sndparam \fstparam$ and
  \item \label{item:lem:estimation:2} $\displaystyle\frac{\sndparam\fstparam\p{2 - \sndparam}\p{2 - \fstparam}}{\p{\sndparam + \fstparam - \sndparam
        \fstparam}^2} \leq 1$.
  \end{enumerate}
\end{lemma}
\begin{proof}
  By assumption, $\sndparam \p{2 - \sndparam} > 0$, i.e., $\sndparam > \frac{1}{2}
  \sndparam^2$, and the same holds for $\fstparam$. Therefore,
  \[
    \sndparam + \fstparam > \frac{1}{2} \sndparam^2 + \frac{1}{2} \fstparam^2 \geq \sndparam
    \fstparam,
  \]
  whence \eqref{item:lem:estimation:1}.
  
  For \eqref{item:lem:estimation:2}, use the inequality $2\sndparam \fstparam \leq \sndparam^2 + \fstparam^2$ in
  \begin{align*}
    \sndparam \fstparam\p{2 - \sndparam}\p{2 - \fstparam}
    &= 4\sndparam\fstparam - 2\sndparam^2 \fstparam - 2\sndparam \fstparam^2 + \sndparam^2 \fstparam^2 \\
    & \leq \sndparam^2 + \fstparam^2 + 2\sndparam \fstparam - 2\sndparam^2 \fstparam - 2\sndparam \fstparam^2 + \sndparam^2 \fstparam^2 \\
    &= \p{\sndparam + \fstparam - \sndparam \fstparam}^2.& \qedhere
  \end{align*}
\end{proof}

\begin{lemma}\label{lem:quadratic_forms}
  Define the quadratic
  forms $Q_1, Q_2\colon \PrimS \times \DualS \to \Real$ by
  \begin{align*}
    Q_1\p{x, y} &= \frac{1}{2\tau \fstparam}\norm{x}^2 + \frac{1}{2\sigma \sndparam}
                  \norm{y}^2 - \frac{1}{\sndparam} \inpr{y}{Lx}, \\
    Q_2\p{x, y} &= \frac{2 - \fstparam}{2\tau} \norm{x}^2 + \frac{2 - \sndparam}{2\sigma} \norm{y}^2 - \frac{\sndparam + \fstparam - \sndparam\fstparam}{\sndparam} \inpr{y}{Lx}
  \end{align*}
  for all $x\in \PrimS$ and $y\in \DualS$. 
  Under the assumptions in Theorem~\ref{thm:convergence}, there exist $C_1, C_2, D_1, D_2 > 0$ such that
  \[
    Q_i\p{x, y} \geq C_i \norm{x}^2 \quad\text{and}\quad Q_i\p{x, y} \geq D_i \norm{y}^2
  \]
  for all $x\in \PrimS$, $y\in \DualS$ and $i=1,2$.
\end{lemma}

\begin{proof}
  We can rewrite
  \begin{align*}
    Q_1\p{x, y} &= \frac{1}{2\sigma \sndparam} \norm{y - \sigma Lx}^2 + \frac{1}{2\tau \fstparam} \norm{x}^2 - \frac{\sigma}{2\sndparam} \norm{Lx}^2, \\
    Q_1\p{x, y} &= \frac{1}{2\tau \fstparam} \norm{x - \frac{\fstparam \tau}{\sndparam} L^* y}^2 + \frac{1}{2\sigma \sndparam} \norm{y}^2 - \frac{\fstparam \tau}{2\sndparam^2} \norm{L^* y}^2
  \end{align*}
  and
  \begin{align*}
    Q_2\p{x, y} &= \frac{2 - \sndparam}{2\sigma} \norm{y - \frac{\sigma\p{\sndparam + \fstparam - \sndparam\fstparam}}{\sndparam\p{2-\sndparam}} Lx}^2 \\
    &\qquad + \frac{2 -\fstparam}{2\tau} \norm{x}^2 - \frac{\sigma\p{\sndparam + \fstparam - \sndparam\fstparam}^2}{2\sndparam^2\p{2-\sndparam}} \norm{Lx}^2, \\
    Q_2\p{x, y} &= \frac{2 - \fstparam}{2\tau} \norm{x - \frac{\tau\p{\sndparam + \fstparam - \sndparam \fstparam}}{\sndparam \p{2 - \fstparam}} L^* y}^2 \\
    &\qquad + \frac{2 - \sndparam}{2\sigma} \norm{y}^2 - \frac{\tau\p{\sndparam + \fstparam - \sndparam \fstparam}^2}{2\sndparam^2 \p{2 - \fstparam}} \norm{L^*y}^2.
  \end{align*}
  From this, the assertion of the lemma is clear with the quantities
  \begin{align*}
    C_1 &= \frac{1}{2\tau \fstparam} - \frac{\sigma}{2\sndparam} \norm{L}^2 = \frac{\sndparam - \fstparam\sigma \tau \norm{L}^2}{2\tau \sndparam \fstparam}, \\
    D_1 &= \frac{1}{2\sigma \sndparam} - \frac{\fstparam \tau}{2\sndparam^2} \norm{L}^2 = \frac{\sndparam - \fstparam \sigma \tau \norm{L}^2}{2\sigma \sndparam^2}, \\
    C_2 &= \frac{2 -\fstparam}{2\tau} - \frac{\sigma\p{\sndparam + \fstparam - \sndparam\fstparam}^2}{2\sndparam^2\p{2 - \sndparam}} \norm{L}^2 \\
        &= \frac{\sndparam^2\p{2 - \sndparam} \p{2 - \fstparam} - \p{\sndparam + \fstparam - \sndparam\fstparam}^2 \sigma\tau \norm{L}^2}{2\tau \sndparam^2\p{2 - \sndparam}}, \\
    D_2 &= \frac{2 - \sndparam}{2\sigma} - \frac{\tau\p{\sndparam + \fstparam - \sndparam \fstparam}^2}{2\sndparam^2 \p{2 - \fstparam}} \norm{L}^2 \\
    &= \frac{\sndparam^2 \p{2 - \sndparam} \p{2 - \fstparam} - \p{\sndparam + \fstparam - \sndparam \fstparam}^2\sigma\tau \norm{L}^2}{2\sigma \sndparam^2 \p{2 - \fstparam}}
  \end{align*}
  provided that the numerators are positive, i.e.,
  \begin{align*}
    \sigma \tau \norm{L}^2 &< \min\set{\frac{\sndparam}{\fstparam}, \frac{\sndparam^2 \p{2 - \sndparam} \p{2 - \fstparam}}{\p{\sndparam + \fstparam - \sndparam \fstparam}^2}}.
  \end{align*}
  Now, by Lemma~\ref{lem:estimation}, the minimum is always attained by the second value, and
  positivity is guaranteed by \eqref{eq:sigma_tau_bound}.
\end{proof}

\begin{proposition}\label{lem:stability}
  Define $Q_1$ and $Q_2$ as in Lemma~\ref{lem:quadratic_forms}, let $\p{\bar x, \bar y} \in \PrimS \times \DualS$ satisfy \eqref{eq:inclusion_problem}, 
  and let the sequence $\p{q_n, y_n, p_n, x_n}$ be generated by scheme \eqref{eq:stability}. 
  Under the assumptions in Theorem~\ref{thm:convergence}, we have for all $n\geq 1$
  \begin{align*}
    Q_1\p{x_{n+1} - \bar x, y_{n+2} - \bar y} - Q_1\p{x_n - \bar x, y_{n+1} - \bar y} \leq -Q_2\p{p_n - x_n, q_{n+1} - y_{n+1}}.
  \end{align*}
\end{proposition}
\begin{proof}
  Let $\p{\bar x, \bar y}$ satisfy \eqref{eq:inclusion_problem}. Then
  \begin{multline}\label{eq:quadratic_verbose}
    \norel Q_1\p{x_{n+1} - \bar x, y_{n+2} - \bar y} - Q_1\p{x_n - \bar x, y_{n+1} - \bar y} 
    \\ \shoveleft{\qquad
     = \frac{1}{2\tau \fstparam}\p{\norm{x_{n+1} - \bar x}^2 - \norm{x_n - \bar x}^2} 
    }\\ \shoveleft{\qquad\qquad
      + \frac{1}{2\sigma \sndparam}\p{\norm{y_{n+2} - \bar y}^2 - \norm{y_{n+1} - \bar y}^2}  
    }\\ \shoveleft{\qquad\qquad
      + \frac{1}{\sndparam} \p{\inpr{y_{n+1} - \bar y}{Lx_n - L\bar x} - \inpr{y_{n+2} - \bar y}{Lx_{n+1} - L\bar x}} 
    }\\ \shoveleft{\qquad
    = \frac{1}{2\tau \fstparam}\p{\norm{x_n - \bar x + \fstparam\p{p_n - x_n}}^2 - \norm{x_n - \bar x}^2}
    }\\ \shoveleft{\qquad\qquad
       + \frac{1}{2\sigma \sndparam} \p{\norm{y_{n+1} - \bar y + \sndparam\p{q_{n+1} - y_{n+1}}}^2 - \norm{y_{n+1} - \bar y}^2} 
    }\\ \shoveleft{\qquad\qquad
       + \frac{1}{\sndparam} \Bigl( \inpr{y_{n+1} - \bar y}{Lx_n - L\bar x} 
    }\\ \shoveleft{\qquad\qquad
       - \inpr{y_{n+1} - \bar y + \sndparam\p{q_{n+1} - y_{n+1}}}{Lx_n - L\bar x + \fstparam\p{Lp_n - Lx_n}} \Bigr) 
    }\\ \shoveleft{\qquad
     = \frac{\fstparam}{2\tau} \norm{p_n - x_n}^2 + \frac{1}{\tau} \inpr{x_n - \bar x}{p_n - x_n} + \frac{\sndparam}{2\sigma} \norm{q_{n+1} - y_{n+1}}^2 
    }\\ \shoveleft{\qquad\qquad
       + \frac{1}{\sigma} \inpr{y_{n+1} - \bar y}{q_{n+1} - y_{n+1}} + \frac{\fstparam}{\sndparam} \inpr{\bar y - y_{n+1}}{Lp_n - Lx_n} 
    }\\ \shoveleft{\qquad\qquad
       + \inpr{q_{n+1} - y_{n+1}}{L\bar x - Lx_n} + \fstparam \inpr{y_{n+1} -
         q_{n+1}}{Lp_n - Lx_n}.}\\[-15.5pt] 
  \end{multline}
  To estimate the above, we use the monotonicity of the operator $\MonotoneOpB^{-1}$
  together with the inclusions $L\bar x \in \MonotoneOpB^{-1} \bar y$ from \eqref{eq:inclusion_problem} and
  \begin{equation}
    \frac{y_{n+1} - q_{n+1}}{\sigma} + Lp_n + \frac{\fstparam}{\sndparam}\p{Lp_n - Lx_n} \in \MonotoneOpB^{-1} q_{n+1},\label{eq:iteration_inclusion_B}
  \end{equation}
  which is a reformulation of \eqref{eq:stability_q} with $n$ replaced by $n+1$.
  This yields the inequality
  \begin{align}
    0 &\leq \inpr{\frac{y_{n+1} - q_{n+1}}{\sigma} + Lp_n + \frac{\fstparam}{\sndparam}\p{Lp_n - Lx_n} - L\bar x}{q_{n+1} - \bar y} \nonumber \\
      &= \frac{1}{\sigma} \inpr{y_{n+1} - q_{n+1}}{q_{n+1} - \bar y} + \inpr{Lp_n - L\bar x}{q_{n+1} - \bar y} \nonumber \\
      &\qquad + \frac{\fstparam}{\sndparam}\inpr{Lp_n - Lx_n}{q_{n+1} - \bar y} \label{eq:basic_ineq_1}
  \end{align}
  Analogously, we can rewrite \eqref{eq:stability_p} as
  \begin{equation}\label{eq:iteration_inclusion_A}
    \frac{x_n - p_n}{\tau} - L^* y_{n+1} \in \MonotoneOpA p_n.
  \end{equation}
  The monotonicity of $\MonotoneOpA$ together with the inclusion $-L^* \bar y \in \MonotoneOpA \bar x$ from \eqref{eq:inclusion_problem} now yields
  \begin{align}\label{eq:basic_ineq_2}
    0 &\leq \inpr{\frac{x_n - p_n}{\tau} - L^* y_{n+1} + L^* \bar y}{p_n - \bar x} \nonumber \\
       &= \frac{1}{\tau} \inpr{x_n - p_n}{p_n - \bar x} + \inpr{\bar y - y_{n+1}}{Lp_n - L\bar x}. 
  \end{align}
  Adding \eqref{eq:basic_ineq_1} and \eqref{eq:basic_ineq_2} yields
  \begin{multline}\label{eq:monotonicity_inequality_combined}
    0 \leq \frac{1}{\sigma} \inpr{y_{n+1} - q_{n+1}}{q_{n+1} - \bar y} + \inpr{Lp_n - L\bar x}{q_{n+1} - y_{n+1}} \\
    + \frac{\fstparam}{\sndparam}\inpr{Lp_n - Lx_n}{q_{n+1} - \bar y} + \frac{1}{\tau} \inpr{x_n - p_n}{p_n - \bar x}, 
  \end{multline}
  which, combined with \eqref{eq:quadratic_verbose}, gives
  \begin{multline}\label{eq:stability_conclusion}
    \norel Q_1\p{x_{n+1} - \bar x, y_{n+2} - \bar y} - Q_1\p{x_n - \bar x, y_{n+1} - \bar y} 
    \\ \shoveleft{\qquad
    \leq \frac{\fstparam}{2\tau} \norm{p_n - x_n}^2 + \frac{1}{\tau} \inpr{x_n - \bar x}{p_n - x_n} + \frac{\sndparam}{2\sigma} \norm{q_{n+1} - y_{n+1}}^2 
    }\\ \shoveleft{\qquad\qquad
     + \frac{1}{\sigma} \inpr{y_{n+1} - \bar y}{q_{n+1} - y_{n+1}} + \frac{\fstparam}{\sndparam} \inpr{\bar y - y_{n+1}}{Lp_n - Lx_n} 
    }\\ \shoveleft{\qquad\qquad
     + \inpr{q_{n+1} - y_{n+1}}{L\bar x - Lx_n} + \fstparam \inpr{y_{n+1} - q_{n+1}}{Lp_n - Lx_n} 
    }\\ \shoveleft{\qquad\qquad
     + \frac{1}{\sigma} \inpr{y_{n+1} - q_{n+1}}{q_{n+1} - \bar y} + \inpr{Lp_n - L\bar x}{q_{n+1} - y_{n+1}} 
    }\\ \shoveleft{\qquad\qquad
     + \frac{\fstparam}{\sndparam}\inpr{Lp_n - Lx_n}{q_{n+1} - \bar y} + \frac{1}{\tau} \inpr{x_n - p_n}{p_n - \bar x} 
    }\\ \shoveleft{\qquad
    = \p{\frac{\fstparam}{2\tau} - \frac{1}{\tau}} \norm{p_n - x_n}^2 + \p{\frac{\sndparam}{2\sigma} - \frac{1}{\sigma}} \norm{q_{n+1} - y_{n+1}}^2 
    }\\ \shoveleft{\qquad\qquad
     + \p{\frac{\fstparam}{\sndparam} + 1 - \fstparam} \inpr{q_{n+1} - y_{n+1}}{Lp_n - Lx_n} 
    }\\ \shoveleft{\qquad
      = -Q_2\p{p_n - x_n, q_{n+1} - y_{n+1}}.} \\[-15.5pt]
  \end{multline}
  This concludes the proof.
\end{proof}

\begin{proposition}\label{lem:stability_strongly_monotone}
  Let $Q_1$ and $Q_2$ be defined as in Lemma~\ref{lem:quadratic_forms} and assume the conditions stated in Theorem~\ref{thm:convergence} hold.
  \begin{enumerate}
  \item If $\MonotoneOpA$ is $\mu_1$-strongly monotone for some $\mu_1 > 0$, then
  \begin{multline*}
    Q_1\p{x_{n+1} - \bar x, y_{n+2} - \bar y} - Q_1\p{x_n - \bar x, y_{n+1} -
      \bar y} + \mu_1 \norm{p_n - \bar x}^2 \\
    \leq -Q_2\p{p_n - x_n, q_{n+1} - y_{n+1}}.
  \end{multline*}
  \item If $\MonotoneOpB^{-1}$ is $\mu_2$-strongly monotone for some $\mu_2 > 0$, then
  \begin{multline*}
      Q_1\p{x_{n+1} - \bar x, y_{n+2} - \bar y} - Q_1\p{x_n - \bar x, y_{n+1} -
        \bar y} + \mu_2 \norm{q_{n+1} - \bar y}^2 \\
      \leq -Q_2\p{p_n - x_n, q_{n+1} - y_{n+1}}.
    \end{multline*}
  \end{enumerate}
\end{proposition}
\begin{proof}
  If $\MonotoneOpA$ is $\mu_1$-strongly monotone, we obtain from
  \eqref{eq:iteration_inclusion_A} and $-L^* \bar y\in \MonotoneOpA\bar x$ \eqref{eq:inclusion_problem} the
  estimation
  \[
    \mu_1 \norm{\bar x - p_n}^2 \leq \inpr{\frac{x_n - p_n}{\tau} - L^* y_{n+1}
      + L^* \bar y}{p_n - \bar x},
  \]
  which is a sharpened version of \eqref{eq:basic_ineq_2}. By modifying
  \eqref{eq:monotonicity_inequality_combined} and
  \eqref{eq:stability_conclusion} accordingly, we get the assumption.
  The case of a strongly monotone $\MonotoneOpB^{-1}$ is analogously shown by improving
  \eqref{eq:basic_ineq_1}.
\end{proof}

Having stated and proved the necessary estimations, we are now ready to prove Theorem~\ref{thm:convergence}.
\begin{proof}[Proof of Theorem~\ref{thm:convergence}]
  Let $\p{\bar x, \bar y}$ satisfy \eqref{eq:inclusion_problem}. By
  Proposition~\ref{lem:stability}, we get the estimation
  \begin{align*}
    Q_1\p{x_{n+1} - \bar x, y_{n+2} - \bar y} - Q_1\p{x_n - \bar x, y_{n+1} - \bar y} \leq -Q_2\p{p_n - x_n, q_{n+1} - y_{n+1}}.
  \end{align*}
  Considering Lemma~\ref{lem:quadratic_forms}, we see that the real sequence
  \[
    \p{Q_1\p{x_n - \bar x, y_{n+1} - \bar y}}_{n\geq 1}
  \]
  is monotonically nonincreasing and therefore has a limit for each primal-dual
  solution $\p{\bar x, \bar y}$. Furthermore, for all $N \geq 1$,
  \begin{multline*}
    Q_1\p{x_N - \bar x, y_{N+1} - \bar y} - Q_1\p{x_0 - \bar x, y_1 - \bar y} \\
    \leq - \sum_{n=0}^{N-1} Q_2\p{p_n - x_n, q_{n+1} - y_{n+1}}.
  \end{multline*}
  By Lemma~\ref{lem:quadratic_forms}, we have $Q_1\p{x_N - \bar x, y_{N+1} -
    \bar y} \geq 0$ and
  \begin{align*}
    Q_1\p{x_0 - \bar x, y_1 - \bar y} &\geq \sum_{n=0}^{N-1} Q_2\p{p_n - x_n, q_{n+1} - y_{n+1}} \\
    &\geq \sum_{n=0}^{N-1} C_2\norm{p_n - x_n}^2
  \end{align*}
  as well as
  \[
    Q_1\p{x_0 - \bar x, y_1 - \bar y} \geq \sum_{n = 0}^{N-1} D_2 \norm{q_{n+1}
      - y_{n+1}}^2.
  \]
  Since this holds for arbitrary $N \geq 1$, this proves parts
  \eqref{item:thm:convergence:primal_summability} and
  \eqref{item:thm:convergence:dual_summability} of the theorem.

  On the other hand, we have
  \[
    Q_1\p{x_0 - \bar x, y_1 - \bar y} \geq Q_1\p{x_N - \bar x, y_{N+1} - \bar y}
    \geq C_1 \norm{x_N - \bar x}^2
  \]
  and
  \[
    Q_1\p{x_0 - \bar x, y_1 - \bar y} \geq Q_1\p{x_N - \bar x, y_{N+1} - \bar y}
    \geq D_1 \norm{y_{N+1} - \bar y}^2
  \]
  for all $N \geq 1$, so the sequences $\p{x_n}_n$ and $\p{y_n}$ are bounded in
  $\PrimS$ and $\DualS$, respectively. Let $\p{n_k}_k$ be a subsequence with
  $x_{n_k} \weakto x_\infty \in \PrimS$ and $y_{n_k + 1} \weakto y_\infty \in
  \DualS$. By \eqref{eq:iteration_inclusion_A} and
  \eqref{eq:iteration_inclusion_B}, we obtain
  \begin{align*}
    \frac{x_{n_k} - p_{n_k}}{\tau} - L^* y_{n_k + 1} &\in \MonotoneOpA p_{n_k}, \\
    \frac{y_{n_k + 1} - q_{n_k + 1}}{\sigma} + Lp_{n_k} + \frac{\fstparam}{\sndparam} \p{Lp_{n_k} - Lx_{n_k}} &\in \MonotoneOpB^{-1} q_{n_k + 1}.
  \end{align*}
  Now apply \cite[Proposition 2.4]{alotaibi2014solving} with
  \begin{align*}
    a_k &= p_{n_k}, \\
    a_k^* &= \frac{x_{n_k} - p_{n_k}}{\tau} - L^* y_{n_k + 1}, \\
    b_k &= \frac{y_{n_k + 1} - q_{n_k + 1}}{\sigma} + Lp_{n_k} + \frac{\fstparam}{\sndparam}\p{Lp_{n_k} - Lx_{n_k}}, \\
    b_k^* &= q_{n_k + 1}
  \end{align*}
  and observe that
  \begin{align*}
    a_k &= x_{n_k} + \p{p_{n_k} - x_{n_k}} \weakto x_\infty, \\
    b_k^* &= y_{n_k + 1} + \p{q_{n_k + 1} - y_{n_k + 1}} \weakto y_\infty, \\
    a_k^* + L^* b_k^* &= \frac{x_{n_k} - p_{n_k}}{\tau} + L^* \p{q_{n_k + 1} - y_{n_k + 1}} \to 0, \\
    La_k - b_k &= - \frac{y_{n_k + 1} - q_{n_k + 1}}{\sigma} - \frac{\fstparam}{\sndparam}\p{Lp_{n_k} - Lx_{n_k}} \to 0
  \end{align*}
  because parts \eqref{item:thm:convergence:primal_summability} and
  \eqref{item:thm:convergence:dual_summability} imply that $x_{n_k} - p_{n_k} \to
  0$ and $y_{n_k + 1} - q_{n_k + 1} \to 0$ as $k\to +\infty$. This gives
  $Lx_\infty \in \MonotoneOpB^{-1} y_\infty$ and $-L^* y_\infty \in \MonotoneOpA x_\infty$, i.e., $\p{x_\infty, y_\infty}$ satisfies \eqref{eq:inclusion_problem}.
  Since the choice of the weakly convergent
  subsequence was arbitrary, each weak sequential cluster point 
  satisfies \eqref{eq:inclusion_problem}.
  Claim~\eqref{item:thm:convergence:weak_convergence} now follows from
  \cite[Lemma~2.47]{bauschke2017convex} applied to the norm
  $\sqrt{Q_1\p{\cdot}}$ on the product space $\PrimS \times \DualS$ and to the
  solution set of \eqref{eq:inclusion_problem}.

  Now assume that $\MonotoneOpA$ is $\mu_1$-strongly monotone for some $\mu_1 > 0$. By
  Proposition~\ref{lem:stability_strongly_monotone}, we get the estimation
  \begin{multline*}
    Q_1\p{x_{n+1} - \bar x, y_{n+2} - \bar y} - Q_1\p{x_n - \bar x, y_{n+1} -
      \bar y} + \mu_1 \norm{p_n - \bar x}^2 \\
    \leq -Q_2\p{p_n - x_n, q_{n+1} - y_{n+1}}
  \end{multline*}
  for all $n\geq 0$. Choose $N \geq 1$ and sum up this inequality for $n = 0,
  \ldots, N-1$ to obtain
  \begin{multline*}
    Q_1\p{x_N - \bar x, y_{N+1} - \bar y} - Q_1\p{x_0 - \bar x, y_1 -
      \bar y} + \mu_1 \sum_{n=0}^{N-1} \norm{p_n - \bar x}^2 \\
    \leq - \sum_{n=0}^{N-1} Q_2\p{p_n - x_n, q_{n+1} - y_{n+1}}
  \end{multline*}
  Since the terms $Q_1\p{x_N - \bar x, y_{N+1} - \bar y}$ and $\sum_{n=0}^{N-1}
  Q_2\p{p_n - x_n, q_{n+1} - y_{n+1}}$ are nonnegative by
  Lemma~\ref{lem:quadratic_forms}, we obtain
  \[
     \mu_1 \sum_{n=0}^{N-1} \norm{p_n - \bar x}^2 \leq Q_1\p{x_0 - \bar x, y_1 -
       \bar y}.
  \]
  Analogously, one gets
  \[
    \mu_2 \sum_{n=0}^{N-1} \norm{q_{n+1} - \bar y}^2 \leq Q_1\p{x_0 - \bar x,
      y_1 - \bar y},
  \]
  if $B^{-1}$ is $\mu_2$-strongly monotone, and since $N$ is arbitrary, both sums
  \[
    \sum_{n=0}^\infty \norm{p_n - \bar x}^2 \qquad \text{and} \sum_{n=0}^\infty
    \norm{q_{n+1} - \bar y}^2
  \]
  are finite in the respective cases. The uniqueness of the point $\bar x$ under the assumption of
  strong monotonicity of $\MonotoneOpA$ holds by the fact that we have shown $p_n \to \bar
  x$ for \emph{any} solution $\p{\bar x, \bar y}$ of
  \eqref{eq:inclusion_problem}. An analogous argument for $\bar y$ concludes the
  proof of Claim~\eqref{item:thm:convergence:strong_convergence}.
\end{proof}

\begin{remark}
  We were not able to show the weak convergence of PDHG
  \eqref{eq:alg_CP} for $\theta \neq 1$ with this proof method.
  Indeed, by a straightforward calculation it can be shown that from Fej\'er
  monotonicity with respect to any quadratic form of the sequence $\p{x_n,
    y_{n+1}}_n$ the conditions
  \eqref{eq:a11andeq:c11} can be derived, which implies $\theta = 1$.
\end{remark}

\subsection{Application to convex optimization}\label{subsec:convex_opt_solver}
In this section, we specialize the scheme \eqref{eq:stability} to the case where
the monotone operators $A$ and $B$ are subdifferentials $\partial F$ and
$\partial G$ of proper, convex and lower
semicontinuous functions $F: \PrimS \to \RealExt$ and $G: \DualS \to \RealExt$,
resectively. Algorithm \ref{alg:mon_op} then reads as follows:
\begin{algorithm}
  Choose parameters $\sigma > 0$, $\tau > 0$ and $\sndparam \in \Real$, $\fstparam \in
  \Real$ and starting points $x_0 \in \PrimS$, $x_1 \in \PrimS$, $p_0 \in
  \PrimS$, $y_1 \in \DualS$. For all $n = 1, 2, \ldots$, calculate
  \begin{subequations}\label{eq:opt_alg}
    \begin{align}
      q_n &= \Prox^{\sigma}_{G^*} \p{y_n + \sigma L \p{p_{n-1} + \frac{\fstparam}{\sndparam} \p{p_{n-1} - x_{n-1}}}}, \label{eq:opt_alg_q}\\
      y_{n+1} &= y_n + \sndparam \p{q_n - y_n}, \label{eq:opt_alg_y} \\
      p_n &= \Prox^{\tau}_{F} \p{x_n - \tau L^* y_{n+1}}, \label{eq:opt_alg_p} \\
      x_{n+1} &= x_n + \fstparam \p{p_n -  x_n}. \label{eq:opt_alg_x}
    \end{align}
  \end{subequations}
  Then, $x_n \weakto \bar x$, $p_n \weakto \bar x$, $y_n \weakto \bar y$, and $q_n
  \weakto \bar y$, where $\p{\bar x, \bar y}$ is a solution of
  \eqref{eq:primal_dual_solutions}, provided that $0 < \sndparam < 2$, $0 < \fstparam < 2$,
  and \eqref{eq:sigma_tau_bound} are satisfied.
\end{algorithm}

In this case, it is possible to get estimations for the Lagrangian, which is
defined in \eqref{eq:saddle_prob_cp}.

\begin{theorem}\label{thm:Lagrangian_rate}
  Given the assumptions in Theorem~\ref{thm:convergence},
  let $F: \PrimS \to
  \RealExt$ and $G: \DualS \to \RealExt$ be two proper, convex and lower
  semicontinuous functions. Let $x \in
  \PrimS$ and $y \in \DualS$ be arbitrary. Then the
  sequence $\p{q_n, y_n, p_n x_n}$ generated by \eqref{eq:opt_alg} satisfies
  \begin{align*}
    \min_{n=0, \ldots, N-1} \p{\Lagr{p_n}{y} - \Lagr{x}{q_{n+1}}} &\leq \frac{1}{N} Q_1\p{x_0 - x, y_1 - y}, \\
    \Lagr{\frac{1}{N} \sum_{n=0}^{N-1} p_n}{y} - \Lagr{x}{\frac{1}{N} \sum_{n=0}^{N-1} q_{n+1}} &\leq \frac{1}{N} Q_1\p{x_0 - x, y_1 - y}.
  \end{align*}
\end{theorem}

This theorem is proved using the following proposition, which bounds the Lagrangian in
terms of the quadratic forms defined in Lemma~\ref{lem:quadratic_forms}.
\begin{proposition}\label{lem:stability_lagrangian}
  Given the assumptions in Theorem~\ref{thm:convergence}, let $F\colon \PrimS \to \RealExt$ and $G\colon \DualS \to \RealExt$ be two
  proper, convex and lower semicontinuous functions. Let $x \in \PrimS$ and
  $y\in \DualS$ be arbitrary. Then the sequence $\p{q_n, y_n, p_n, x_n}$
  generated by \eqref{eq:opt_alg} satisfies
  \begin{multline*}
    \Lagr{p_n}{y} - \Lagr{x}{q_{n+1}} \leq Q_1\p{x_n - x, y_{n+1} - y} -
    Q_1\p{x_{n+1} - x, y_{n+2} - y} \\ - Q_2\p{p_n - x_n, q_{n+1} - y_{n+1}}
  \end{multline*}
  for all $n \geq 1$, $x\in \PrimS$ and $y\in \DualS$.
\end{proposition}
\begin{proof}
  Since $\MonotoneOpB^{-1} = \partial G^*$ and $\MonotoneOpA = \partial F$, the inclusions
  \eqref{eq:iteration_inclusion_B} and \eqref{eq:iteration_inclusion_A} provide
  certain subgradients, which imply the inequalities
  \begin{align*}
    G^*\p{y} &\geq G^*\p{q_{n+1}} + \frac{1}{\sigma} \inpr{y_{n+1} - q_{n+1}}{y - q_{n+1}} + \inpr{Lp_n}{y - q_{n+1}} \\
    &\qquad + \frac{\fstparam}{\sndparam} \inpr{Lp_n - Lx_n}{y - q_{n+1}}, \\
    F\p{x} &\geq F\p{p_n} + \frac{1}{\tau} \inpr{x_n - p_n}{x - p_n} - \inpr{L^* y_{n+1}}{x - p_n}.
  \end{align*}
  Therefore, we have
  \begin{multline*}
    \Lagr{p_n}{y} - \Lagr{x}{q_{n+1}}
    \\ \shoveleft{\qquad
      = \inpr{Lp_n}{y} + F\p{p_n} - G^*\p{y} - \inpr{Lx}{q_{n+1}} - F\p{x} + G^*\p{q_{n+1}} 
    }\\ \shoveleft{\qquad
    \leq \frac{1}{\tau} \inpr{x_n - p_n}{p_n - x} + \inpr{Lp_n - Lx}{q_{n+1} - y_{n+1}} 
    }\\ \shoveleft{\qquad\qquad
      + \frac{1}{\sigma} \inpr{y_{n+1} - q_{n+1}}{q_{n+1} - y} +
      \frac{\fstparam}{\sndparam}\inpr{Lp_n - Lx_n}{q_{n+1} - y}.} \\[-15.5pt]
  \end{multline*}
  
  The right-hand side is now (except for the replacement of $\bar x$ and $\bar
  y$ by $x$ and $y$, respectively) equal to the one in
  \eqref{eq:monotonicity_inequality_combined}, and one easily checks by an
  analogous calculation, that it equals the expression in the assertion.
\end{proof}

\begin{proof}[Proof of Theorem \ref{thm:Lagrangian_rate}]
  By summing the inequality in Proposition~\ref{lem:stability_lagrangian} for $n = 0,
  \ldots, N-1$ and dividing by $N$ for some $N \geq 1$, we get
  \[
    \frac{1}{N} \sum_{n=0}^{N-1} \p{\Lagr{p_n}{y} - \Lagr{x}{q_{n+1}}} \leq
    \frac{1}{N} Q_1\p{x_0 - x, y_1 - y}
  \]
  for all $x \in \PrimS$ and $y \in \DualS$, where we dropped nonpositive terms
  on the right-hand side.

  We have two possibilities to further estimate the left-hand side: First, we
  notice that it is the arithmetic mean of numbers, which is always greater than
  the minimum, i.e.,
  \[
    \frac{1}{N} \sum_{n=0}^{N-1} \p{\Lagr{p_n}{y} - \Lagr{x}{q_{n+1}}} \geq
    \min_{n=0, \ldots, N-1} \p{\Lagr{p_n}{y} - \Lagr{x}{q_{n+1}}}.
  \]
  On the other hand, the Lagrangian is convex in its first and concave in its
  second component, so
  \[
    \frac{1}{N} \sum_{n=0}^{N-1} \p{\Lagr{p_n}{y} - \Lagr{x}{q_{n+1}}} \geq
    \Lagr{\frac{1}{N} \sum_{n=0}^{N-1} p_n}{y} - \Lagr{x}{\frac{1}{N}
      \sum_{n=0}^{N-1} q_{n+1}}. \qedhere
  \]
\end{proof}

\section{Learning an optimization solver}\label{sec:learning_opt_solver}
Most optimization problems are solved using iterative methods, akin to the ones
presented in Sections~\ref{sec:background} and \ref{sec:new_solver}.
However, the number of iterations it takes in order for the algorithm to converge is
in general hard to predict, which creates problems in time-critical
applications.
In these situations one could instead consider only doing a
predefined fixed number $n$ of iterations.
A natural question that arises in response to this is:
\emph{what parameter values in the optimization solver give the best improvement of
the objective function in $n$ iterations}?
This question leads to a meta-optimization over optimization solvers. 
Moreover, in general we are not only interested in optimizing one single cost
function, but rather a (potentially infinite) family $\{ F_{\theta} \}_{\theta
  \in \Theta}$
of cost functions, each with a minimizer $\bar x_\theta$. 
Hence, to make the question precise one needs to specify which family of
optimization solvers one is considering, which is the family of cost functions
of interested, and what is meant with \enquote{best improvement}.

One such question was raised in
\cite{drori2014performance}, where the authors consider the worst-case
performance $\sup_{\theta \in \Theta} \bigl[ F_\theta(x_n) - F_\theta(\bar{x}_\theta) \bigr]$ of gradient-based
algorithms over the set of continuously differentiable functions with
Lipschitz-continuous gradients, and with a uniform upper bound on the
Lipschitz constants.
Subsequent work
along the same lines can be found in \cite{kim2016optimized, taylor2017smooth}.

The idea of optimizing over optimization solvers has also been considered from
a machine learning perspective.
This has for example been done using \emph{reinforcement learning}
\cite{li2016learning}, and using \emph{unsupervised learning}
\cite{gregor2010learning, andrychowicz2016learning}.
In the latter category, one looks for algorithm parameters which minimize the expected value of the
difference in objective function value,
\begin{equation}\label{eq:training_cost_function}
  \Expect_{\theta}\big[ F_\theta(x_n) - F_\theta(\bar x_\theta) \big] =
  \Expect_{\theta} \big[ F_\theta(x_n)\big] - \Expect_{\theta} \big[ F_\theta(\bar x_\theta) \big]
\end{equation}
where $\Theta$ is endowed with a probability measure and $x_n$ is the output of the
algorithm after $n$ iterations. However,
optimizing \eqref{eq:training_cost_function} with respect to the parameters of the method is
independent of the optimal points $\{ \bar{x}_{\theta} \}_{\theta \in \Theta}$, thus, this translates into unsupervised learning,
i.e., the cost function $\Expect_{\theta} \big[ F_\theta(x_n)\big]$ does not
depend on $\bar x_\theta$.
In this setting, \cite{andrychowicz2016learning} restricts attention to 
an architecture that operates individually on each
coordinate of $x$.
This is done in order to limit the number of parameters in the algorithm, which
otherwise would grow exponentially with the dimension of $x$.
To overcome this, we use an approach similar to \cite{gregor2010learning},
where the network architecture is inspired by modern first-order optimization
solvers for nonsmooth problems, as presented in Sections~\ref{sec:background}
and \ref{sec:new_solver}.
Similar ideas have also recently been explored for supervised learning in inverse
problems in \cite{YaSuLiXu16, adler2017solving, adler2017learned, mardani2017deep, putzky2017recurrent, adler2017learning, hammernik2018learning}.

\subsection{Unrolled gradient descent as a neural
  network}\label{sec:gradient_descent}
Before we define the architecture considered in this work, we first present
an illustrative example.
To this end, consider the optimization problem
\[
\min_x \; \; F(x).
\]
We assume that $F$ is smooth, which means that the problem can be solved using
a standard gradient descent algorithm, i.e., by performing the updates
\[
x_{k} = x_{k - 1} - \StepLength_k \nabla F(x_k).
\]
The gradient descent algorithm contains a set of parameters that need to be
selected, namely the step length for each iteration, $\StepLength_k$.
This is normally done via the Goldstein rule or backtracking line search
(Armijo rule) \cite{bertsekas1999nonlinear}, which under suitable conditions
ensures convergence to the optimal point $\bar x$.

However, if we only run the algorithm for a fixed number $n$ of steps,
the gradient descent algorithm can be seen as a \emph{feedforward neural network},
as shown in Figure~\ref{fig:gd_graph}.
Each layer in the network performs the computation
$x_{k - 1} - \StepLength_k \nabla F(x_{k-1})$ and the parameters of the network
are $[\StepLength_1, \ldots, \StepLength_n]$.
Moreover, if the step length is fixed to be the same in all iterations, i.e., $\StepLength_1 = \ldots = \StepLength_n = \StepLength$ for some $\StepLength$,
the gradient descent algorithm can in fact be interpreted as a \emph{recurrent neural network}.
In both cases,
for a given family $\{ F_\theta \}_{\theta \in \Theta}$ of cost functions
the network parameter(s) can be trained (optimized) by minimizing
$\Expect_{\theta}\big[ F_\theta(x_n)]$, 
where $x_n$ is the output of the network in Figure~\ref{fig:gd_graph}.
For simple cases this can be done analytically.

\begin{figure}
\centering
\tikzset{%
  block/.style    = {draw, thick, rectangle, minimum height = 2em,
    minimum width = 2em},
  circ/.style      = {draw, circle, node distance = 1.5cm},
}

\begin{tikzpicture}[auto, node distance=1.37cm, >=triangle 45]

\draw
node [circ] (x0) {\large $x_0$}
node [block, below right of=x0] (grad1) {\large $\nabla F(\cdot)$}
node [block, right of=grad1] (mult1) {\large $\times$}
node [block, below of=mult1] (alpha1) {\large $\sigma_1$}
node [block, above right of=mult1] (sub1) {\large $-$}
node [above of=sub1] (x1) {\large $x_1$};

\draw[->](x0) -- (grad1);
\draw[->](x0) -- (sub1);
\draw[->](grad1) -- (mult1);
\draw[->](alpha1) -- (mult1);
\draw[->](mult1) -- (sub1);
\draw[dashed,->](sub1) -- (x1);

\draw
node [block, below right of=sub1] (grad2) {\large $\nabla F(\cdot)$}
node [right of=sub1] (upperdotsspace) {}
node [right of=upperdotsspace] (upperdots) {\large $\ldots$}
node [right of=grad2] (lowerdots) {\large $\ldots$}
node [block, right of=upperdots] (subnminone) {\large $-$}
node [above of=subnminone] (xnminone) {\large $x_{n-1}$};

\draw[->](sub1) -- (upperdots);
\draw[->](sub1) -- (grad2);
\draw[->](grad2) -- (lowerdots);
\draw[->](grad2) -- (lowerdots);
\draw[->](upperdots) -- (subnminone);
\draw[->](lowerdots) -- (subnminone);
\draw[dashed,->](subnminone) -- (xnminone);

\draw
node [block, below right of=subnminone] (gradn) {\large $\nabla F(\cdot)$}
node [block, right of=gradn] (multn) {\large $\times$}
node [block, below of=multn] (alphan) {\large $\sigma_n$}
node [block, above right of=multn] (subn) {\large $-$}
node [above of=subn, name=xn] {\large $x_n$};

\draw[->](subnminone) -- (gradn);
\draw[->](subnminone) -- (subn);
\draw[->](gradn) -- (multn);
\draw[->](alphan) -- (multn);
\draw[->](multn) -- (subn);
\draw[->](subn) -- (xn);
\end{tikzpicture}
\caption{Gradient descent.}
\label{fig:gd_graph}
\end{figure}
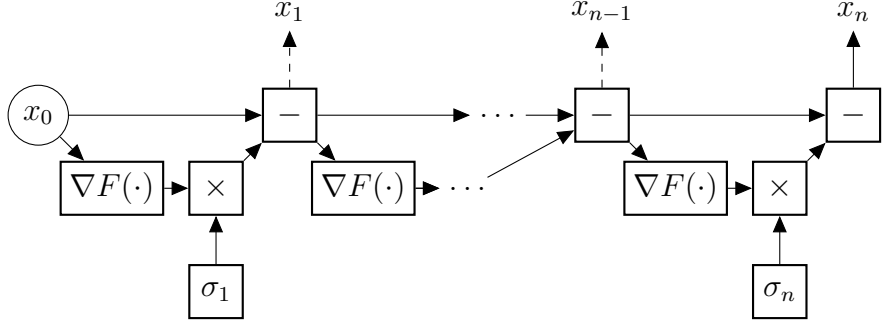

\begin{example}\label{ex:gradient_descent}
  Consider the family $(F_b)_b$ of functions $F_b: \Real^n \to \Real$ given by
  $F_b(x) = \frac{1}{2} x^\top \ExMatrix x - b^\top x$, where $\ExMatrix \in \Real^{n\times n}$ is
  a (fixed) symmetric and positive definite matrix. The minimum of $F_b$ is given
  by $\bar x_b = \ExMatrix^{-1} b$. Denote by $\Lambda_\sigma$ the result of taking a
  gradient step of length $\sigma > 0$, i.e.,
  \[
    \Lambda_\sigma(x) = x - \sigma \nabla F_b(x) = x - \sigma (\ExMatrix x - b),
    \qquad x\in \Real^n.
  \]
  Let $x_0 \in \Real^n$ be an arbitrary starting point of the iteration.
  This gives
  \begin{align*}
    F_b(\Lambda_\sigma(x_0))
    &= F_b(x_0 - \sigma (\ExMatrix x_0 - b)) \\
    &= \frac{1}{2} (x_0 - \sigma (\ExMatrix x_0 - b))^\top \ExMatrix (x_0 - \sigma (\ExMatrix x_0 - b))
      - b^\top (x_0 - \sigma (\ExMatrix x_0 - b)) \\
    &= \frac{\sigma^2}{2} (\ExMatrix x_0 - b)^\top \ExMatrix (\ExMatrix x_0 - b)
      - \sigma \norm{\ExMatrix x_0 - b}^2 + F_b(x_0).
  \end{align*}
  Let $\stb$ be a random variable distributed according to $\stb \sim \Probdist$ for
  some probability distribution $\Probdist$ with finite first and second
  moments.
  Finding a $\sigma$ that minimizes the expectation
  \begin{multline*}
    \Expect_{\stb\sim \Probdist} \Bigl[ F_{\stb}(\Lambda_\sigma(x_0)) \Bigr]
    = \frac{\sigma^2}{2} \Expect_{\stb \sim \Probdist}
    \Bigl[ (\ExMatrix x_0 - \stb)^\top \ExMatrix (\ExMatrix x_0 - \stb) \Bigr] \\
    - \sigma \Expect_{\stb\sim \Probdist} \Bigl[ \norm{\ExMatrix x_0 - \stb}^2
    \Bigr] + \Expect_{\stb\sim \Probdist} \Bigl[ F_\stb (x_0) \Bigr],
  \end{multline*}
  is a quadratic problem in one variable, and the optimal value of $\sigma$ is thus
  \begin{align*}
    \sigma &= \frac{\Expect_{\stb\sim \Probdist} \Bigl[ \norm{\ExMatrix x_0 - \stb}^2
      \Bigr]}{\Expect_{\stb \sim \Probdist} \Bigl[ (\ExMatrix x_0 - \stb)^\top \ExMatrix (\ExMatrix x_0 - \stb)
      \Bigr]} \\
    &= \frac{\norm{\ExMatrix x_0}^2 - 2 (\ExMatrix x_0)^\top \Expect_{\stb\sim \Probdist}
      \bigl[ \stb \bigr] + \Expect_{\stb\sim \Probdist} \Bigl[ \| \stb \|^2 \Bigr]}
      {x_0^\top \ExMatrix^3 x_0 - 2(\ExMatrix^2 x_0)^\top \Expect_{\stb\sim \Probdist} \bigl[
      \stb \bigr] + \Expect_{\stb\sim \Probdist} \Bigl[ \stb^\top \ExMatrix \stb \Bigr]}.
  \end{align*}
  In some particular cases this expression can be simplified.
  For example if $\ExMatrix = I$, then $\sigma = 1$ as expected.
  Or if $x_0 = 0$, then $\sigma = \Expect_{\stb\sim \Probdist} [ \| \stb \|^2 ] / \Expect_{\stb\sim \Probdist} [ \stb^\top \ExMatrix \stb ]$.
\end{example}

\subsection{Parametrizing a family of optimization algorithms}\label{subsec:param_learned_opt_solver}
Similarly to the considerations in Section \ref{sec:gradient_descent}, for a fixed
number of iterations one can
consider the optimization algorithms \eqref{eq:alg_CP}, \eqref{eq:alg_DR} and \eqref{eq:opt_alg} as neural networks,
where the variables we want to train are the parameters of the optimization methods.
Optimizing these parameters with respect to the constraints corresponding to each
algorithm is effectively trying to find optimal parameters for the corresponding
algorithm for a given family of cost functions.
However, if one only intends to do a finite number of iterations one could also remove this constraint,
and thereby enlarge the space of schemes one is optimizing over.

As noted in Section~\ref{sec:new_solver}, all of the above mentioned optimization algorithms
can be written on the form \eqref{eq:2x2_scheme_components}.
That means that optimizing over the parameters in \eqref{eq:2x2_scheme_components}
can be seen as optimizing over a space of schemes that includes all three algorithms.
Now, introducing the intermediate states $w_{n} = \coneone q_n + \conetwo y_n$ and
$v_{n+1} = \aoneone p_n +  \aonetwo x_n$,
and the $2 \times 2$ matrices
$\SchemeMatrixC, \SchemeMatrixB, \SchemeMatrixA, \SchemeMatrixD$, the scheme
\eqref{eq:2x2_scheme_components} can be written as
\begin{equation}\label{eq:2x2_scheme_matrix}
\begin{aligned}
  & \begin{bmatrix}
    w_{n} \\ y_{n+1}
  \end{bmatrix}
  =
  (\SchemeMatrixC \otimes \Id) \; \;
  \diag(\Prox^{\sigma}_{G^*}, \Id) \;\;
  (\SchemeMatrixB \otimes \Id) \;
  \begin{bmatrix}
    L v_{n} \\ y_{n}
  \end{bmatrix}
  \\
  & \begin{bmatrix}
    v_{n+1} \\ x_{n+1}
  \end{bmatrix}
  = (\SchemeMatrixA \otimes \Id) \; \;
  \diag(\Prox^{\tau}_F, \Id) \; \;
  (\SchemeMatrixD \otimes \Id) \;
  \begin{bmatrix}
    L^* w_{n} \\ x_{n}
  \end{bmatrix},
\end{aligned}
\end{equation}
where the parameters of the scheme are the elements of the matrices.
Here, by $\otimes$ we denote the Kronecker product, and by $\diag\p{A, B, \ldots}$ we denote the diagonal operator with the operators $A, B, \ldots$ on the diagonal.
Connecting this with the previous optimization algorithms, the PDHG algorithm \eqref{eq:alg_CP}
is obtained by setting
\begin{equation*}
  \SchemeMatrixC = \begin{bmatrix}
    1 & 0 \\
    1 & 0
  \end{bmatrix}\!, \quad
  \SchemeMatrixB = \begin{bmatrix}
    \sigma & 1 \\
    0      & 1
  \end{bmatrix}\!, \quad
  \SchemeMatrixA = \begin{bmatrix}
    1 + \theta & -\theta \\
    1          & 0
  \end{bmatrix}\!, \quad
  \SchemeMatrixD = \begin{bmatrix}
    -\tau  & 1 \\
    0      & 1
  \end{bmatrix}\!,
\end{equation*}
the primal-dual Douglas-Rachford algorithm 
\eqref{eq:alg_DR} by taking
\begin{equation*}
  \SchemeMatrixC = \begin{bmatrix}
    \lambda_n & 1-\lambda_n \\
    \lambda_n & 1-\lambda_n
  \end{bmatrix}\!, \quad
  \SchemeMatrixB = \begin{bmatrix}
    \sigma & 1 \\
    0      & 1
  \end{bmatrix}\!, \quad
  \SchemeMatrixA = \begin{bmatrix}
    2         & -1 \\
    \lambda_n & 1-\lambda_n
  \end{bmatrix}\!, \quad
  \SchemeMatrixD = \begin{bmatrix}
    -\tau  & 1 \\
    0      & 1
  \end{bmatrix}\!,
\end{equation*}
and the proposed algorithm from Section~\ref{sec:new_solver} by setting
\begin{equation*}
  \SchemeMatrixC = \begin{bmatrix}
    \sndparam & 1 - \sndparam \\
    \sndparam & 1 - \sndparam
  \end{bmatrix}\!, \;
  \SchemeMatrixB = \begin{bmatrix}
    \sigma & 1 \\
    0      & 1
  \end{bmatrix}\!, \;
  \SchemeMatrixA = \begin{bmatrix}
    1 + \tfrac{\fstparam}{\sndparam} & -\tfrac{\fstparam}{\sndparam} \\
    \fstparam                   & 1 - \fstparam
  \end{bmatrix}\!, \;
  \SchemeMatrixD = \begin{bmatrix}
    -\tau  & 1 \\
    0      & 1
  \end{bmatrix}\!.
\end{equation*}

Considering \eqref{eq:2x2_scheme_matrix} as a neural network, the structure can easily be extended in order to incorporate more memory in the
network.
In this work we assume that the computationally expensive part of the algorithm is the evaluation of the operator $L$ and its adjoint,
which is typically the case in inverse problems in imaging,
e.g., in three-dimensional CT \cite{natterer2001themathematics, natterer2001mathematical}.
Therefore, the extension presented here thus keeps one evaluation $L$ and one evaluation of $L^*$ in each iteration.

To this end, let $N$ be the number of primal variables $x^1, \ldots, x^N \in \PrimS$ and
$M$ be the number of dual variables $y^1, \ldots, y^M \in \DualS$.
Introducing the
four sequences of matrices $\SchemeMatrixC_n, \SchemeMatrixB_n \in \Real^{M \times M}$ and
$\SchemeMatrixA_n, \SchemeMatrixD_n \in \Real^{N \times N}$,
the iterations in \eqref{eq:2x2_scheme_matrix}
can be extended to yield the following algorithm.
\begin{algorithm}
  Choose parameters $\SchemeMatrixC_n, \SchemeMatrixB_n \in \Real^{M \times M}$ and
  $\SchemeMatrixA_n, \SchemeMatrixD_n \in \Real^{N \times N}$, stepsizes
  $\sigma, \tau > 0$, and starting points
  $x^1_0, \ldots, x^N_0 \in \PrimS$, $y^2_0, \ldots, y^M_0 \in \DualS$. For all $n = 1, 2, \ldots$, calculate
  \begin{equation*}
    \begin{aligned}
      \begin{bmatrix}
        y^1_{n+1} \\ y^2_{n+1} \\ \vdots \\ y^M_{n+1}
      \end{bmatrix} &= (\SchemeMatrixC_n \otimes \Id) \; \; \diag(\Prox^{\sigma}_{G^*}, \Id^{M-1}) \;\;
      (\SchemeMatrixB_n \otimes \Id) \;
      \begin{bmatrix}
        L x^1_n \\ y^2_n \\ \vdots \\ y^M_n
      \end{bmatrix}, \\
      \begin{bmatrix}
        x^1_{n+1} \\ x^2_{n+1} \\ \vdots \\ x^N_{n+1}
      \end{bmatrix} &= (\SchemeMatrixA_n \otimes \Id) \; \; \diag(\Prox^{\tau}_F, \Id^{N-1}) \; \; (\SchemeMatrixD_n \otimes \Id) \;
      \begin{bmatrix}
        L^* y^1_{n+1} \\ x^2_n \\ \vdots \\ x^N_n
      \end{bmatrix}.
    \end{aligned}
  \end{equation*}
\end{algorithm}

\begin{remark}
  For the more general formulation of \eqref{eq:general_form}, more specialized
  network architectures than the one resulting from the choice
  \eqref{eq:product_space_reduction} are possible, which handle the dual spaces
  separately instead of using the same stepsize $\sigma$ and matrices
  $\SchemeMatrixC_n$ and $\SchemeMatrixB_n$ for all of them.
  An alternative network in the spirit of, e.g., \cite[Theorem
  2]{bot2015convergence}, to solve \eqref{eq:general_form} reads as follows.
\end{remark}
\begin{algorithm}
  Choose parameters $\SchemeMatrixC_{n, i}, \SchemeMatrixB_{n, i} \in \Real^{M \times M}$, for $i =
  1, \ldots, m$, and $\SchemeMatrixA_n, \SchemeMatrixD_n \in \Real^{N \times N}$,
  stepsizes $\sigma_1, \ldots, \sigma_m, \tau > 0$, and starting points
  $x^1_0, \ldots, x^N_0 \in \PrimS$, $y^2_{0, i}, \ldots, y^M_{0, i} \in
  \DualS_i$, $i = 1, \ldots, m$. For all $n = 1, 2, \ldots$, calculate
  \begin{equation}\label{eq:really_general_scheme}
    \begin{aligned}
      \begin{bmatrix}
        y^1_{n+1, i} \\ y^2_{n+1, i} \\ \vdots \\ y^M_{n+1, i}
      \end{bmatrix} &= (\SchemeMatrixC_{n, i} \otimes \Id) \; \; \diag(\Prox^{\sigma_i}_{G_i^*}, \Id^{M-1}) \;\;
      (\SchemeMatrixB_{n, i} \otimes \Id) \;
      \begin{bmatrix}
        L_i x^1_n \\ y^2_{n, i} \\ \vdots \\ y^M_{n, i}
      \end{bmatrix}, \\
      &\hspace{8.3cm} i = 1, \ldots, m, \\
      \begin{bmatrix}
        x^1_{n+1} \\ x^2_{n+1} \\ \vdots \\ x^N_{n+1}
      \end{bmatrix} &= (\SchemeMatrixA_n \otimes \Id) \; \; \diag(\Prox^{\tau}_F, \Id^{N-1}) \; \; (\SchemeMatrixD_n \otimes \Id) \;
      \begin{bmatrix}
        \sum_{i = 1}^m L_i^* y^1_{n+1, i} \\ x^2_n \\ \vdots \\ x^N_n
      \end{bmatrix}.
    \end{aligned}
  \end{equation}
\end{algorithm}

\subsubsection{Extension to forward-backward-forward methods}

Some methods in the literature, so called forward-backward-forward methods,
include an extra evaluation of the operator and its adjoint per iteration, see,
e.g., \cite{combettes2012primal, bot2013douglas, briceno2011monotone}.
However, since the evaluation of the linear operator is assumed to be the expensive part in our setting we consider this as two iterations.
Thus, if we start with the $x$-iterate and allow for two iterations in our framework to complete one
iteration in such a framework, our proposed algorithm contains, e.g., \cite[Equation (3.1)]{briceno2011monotone}.
Letting $\cdot$ denote an element that can take any value, one such set of matrices is given by
\begin{align*}
\SchemeMatrixC_{2n} &=
\begin{bmatrix}
0     & 0     & 1     \\
\cdot & \cdot & \cdot \\
1     & 0     & 0 
\end{bmatrix},
& \SchemeMatrixB_{2n} &= 
\begin{bmatrix}
\gamma_n & 0 & 1 \\
1        & 0 & 0 \\
0        & 0 & 1 
\end{bmatrix}, \\
\SchemeMatrixA_{2n} &=
\begin{bmatrix}
1     & 0     & -1     \\
\cdot & \cdot & \cdot \\
1     & 1     & 0 
\end{bmatrix},
& \SchemeMatrixD_{2n} &=
\begin{bmatrix}
-\gamma_n & 0     & 1 \\
\gamma_n  & 0     & 0 \\
0         & 0     & 1 
\end{bmatrix},
\intertext{for the even iterations and}
\SchemeMatrixC_{2n+1} &=
\begin{bmatrix}
0     & 1     & 0 \\
\cdot & \cdot & \cdot \\
0     & 0     & 1
\end{bmatrix},
& \SchemeMatrixB_{2n+1} &=
\begin{bmatrix}
\cdot    & \cdot & \cdot \\
0        & 0     & 1 \\
\gamma_n & 0     & 1
\end{bmatrix}, \\
\SchemeMatrixA_{2n+1} &=
\begin{bmatrix}
0     & 0     & 1 \\
\cdot & \cdot & \cdot \\
0     & 0     & 1
\end{bmatrix},
& \SchemeMatrixD_{2n+1} &=
\begin{bmatrix}
\cdot     & \cdot & \cdot \\
\cdot     & \cdot & \cdot \\
-\gamma_n & 0     & 1 \\
\end{bmatrix},
\end{align*}
for the odd iterations.

\begin{remark}
Other forward-backward-forward methods have been proposed in the literature, some of which are general enough to include the PDHG as a special case \cite{he2012convergence}, or both the PDHG and the Douglas-Rachford algorithm as special cases \cite{latafat2017asymmetric}.
However, these methods include a step-length computation in their updates.
This computation involves evaluating the norm of current iterates, which is not possible to achieve by only doing the linear operations we propose.
Of course, allowing the matrix elements to be nonlinear functions of the states would allow us to incorporate also these methods, however, that is beyond the scope of this article.
\end{remark}

\section{Application to inverse problems and numerical experiments}\label{sec:simulation}
As we briefly outline next, optimization problems of the type in \eqref{eq:general_form} arise when solving ill-posed inverse problems by means of variational regularization.

The goal in an inverse problem is to recover parameters characterizing a system under investigation from indirect observations.
This can be formalized as the task of estimating (reconstructing) model parameters, henceforth called signal, $\truesignal \in \RecSpace$ from indirect observations (data) $\data \in \DataSpace$ where
\begin{equation}\label{eq:InvProb}
\data = \ForwardOp\p{\truesignal} + \noise.  
\end{equation}  
In the above, $\RecSpace$ and $\DataSpace$ are typically Hilbert or Banach spaces, and $\ForwardOp \colon \RecSpace \to \DataSpace$ (forward operator) models how a given signal gives rise to data in absence of noise.
Furthermore, $\noise \in \DataSpace$ is a single sample of a $\DataSpace$-valued random element that represents the noise component of data.

A natural approach for solving \eqref{eq:InvProb} is to minimize a function
$\DataDiscrepancy \colon \RecSpace \to \Real$ (data discrepancy functional) that
quantifies the miss-fit in data space. Since this function needs to incorporate the
aforementioned forward operator $\ForwardOp$ and the data $\data$, it is often
of the form
  \[
    \DataDiscrepancy\p{\signal} \defeq \LogLikelihood\p{\ForwardOp\p{\signal},
      \data} \quad\text{for some $\LogLikelihood: \DataSpace \times \DataSpace
      \to \Real$.}
  \]
If $\LogLikelihood$ is the negative data log-likelihood, then minimizing
$\signal \mapsto \DataDiscrepancy\p{\signal}$ corresponds to finding a maximum
likelihood solution to \eqref{eq:InvProb}.

However, finding a minimizer to $\DataDiscrepancy$ is an ill-posed problem, meaning that a solution (if it exists) is discontinuous with respect to the data $\data$.
Variational regularization addresses this issue by introducing an additional
function $\Regularizer \colon \RecSpace \to \RealExt$ (regularization functional)
that encodes a priori information about $\truesignal$ and penalizes undesirable
solutions \cite{engl2000regularization}.
This results in an optimization problem
\begin{equation}\label{eq:OptProb}
	\min_{\signal \in \RecSpace} \Bigl[ \lambda \DataDiscrepancy\p{\signal}
  + \Regularizer\p{\signal} \Bigr],
\end{equation}  
which from a statistical perspective can be interpreted as trying to find a maximum a posteriori estimate \cite{kaipio2006statistical}. 
A common choice of regularization functional, especially for inverse problems in imaging, is the total variation (TV) regularization $\Regularizer\p{\signal} \defeq \norm{\nabla \signal}_1$, but several more advanced regularizers have also been suggested in the literature, typically exploiting some kind of sparsity using an $L_1$-like norm \cite{bruckstein2009sparse}.

In this section, we consider an inverse problem in computerized tomography.
To this end, let $\ForwardOp$ be the \emph{Radon transform} and consider TV regularization.
This means that we are interested in minimizing
\begin{equation}\label{eq:CT_cost_function}
  H_b\p{x} = \norm{\ForwardOp\p{x} - b}_2^2 + \lambda \norm{\nabla x}_1,
\end{equation}
i.e., a family of objective functions that is parametrized by the data $b$.
This means that we can apply the ideas from Section
\ref{sec:learning_opt_solver} on learning an optimization solver.

\subsection{Implementation and specifications of the training}
We train and evaluate several of the algorithms described in this article on a
clinically realistic data set, namely simulated data from human abdomen CT scans
as provided by Mayo Clinic for the AAPM Low Dose CT Grand Challenge
\cite{mccollough2016tu}.
Examples of two-dimensional phantoms from this data set
are given in Figure~\ref{fig:mayo_phantoms}.
Throughout all examples, the
size of the image $x$ is $512\times 512$ pixels, and the regularization parameter
$\lambda > 0$ is fixed.
The Radon transform $\ForwardOp$ used in this example is sampled according to a
fan-beam geometry \cite{natterer2001themathematics} and the data is generated by applying $\ForwardOp$ to the phantoms and then adding
 $5\%$ white Gaussian noise.
Examples of such data (sinograms) are also shown in Figure~\ref{fig:mayo_phantoms}.
    
Problem \eqref{eq:CT_cost_function} is obtained from \eqref{eq:general_form}
by setting $F\p{x} \defeq 0$ for all $x$, 
\begin{align*}
  L_1x &\defeq \ForwardOp\p{x},& L_2x &\defeq \nabla x,  \\
  G_1\p{y_1} &\defeq \norm{y_1 - b}_2^2,& G_2\p{y_2} &\defeq \norm{y_2}_1,
\end{align*}
and all the proximal operators are implemented in ODL \cite{adler2017ODL}. If
not stated otherwise, we use \eqref{eq:product_space_reduction} to reduce
\eqref{eq:general_form} to \eqref{eq:primal_prob_cp}.

\begin{figure}
  \centering	
  \begin{subfigure}[t]{.32\linewidth}
    \includegraphics[width=\linewidth]{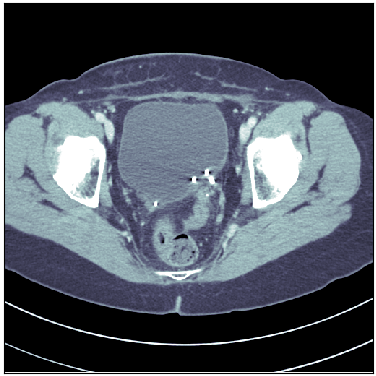}
  \end{subfigure}
  \begin{subfigure}[t]{.32\linewidth}
    \includegraphics[width=\linewidth]{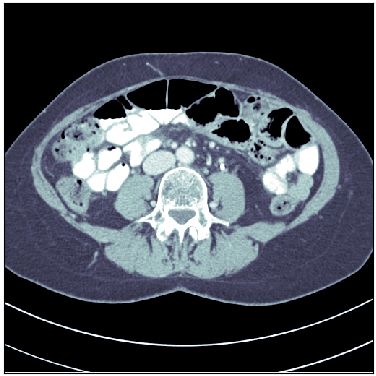}
  \end{subfigure}
  \begin{subfigure}[t]{.32\linewidth}
    \includegraphics[width=\linewidth]{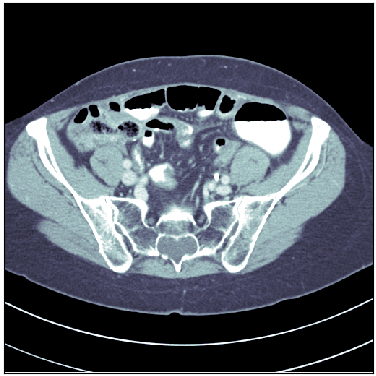}
  \end{subfigure}
  \begin{subfigure}[t]{.32\linewidth}
    \includegraphics[width=\linewidth]{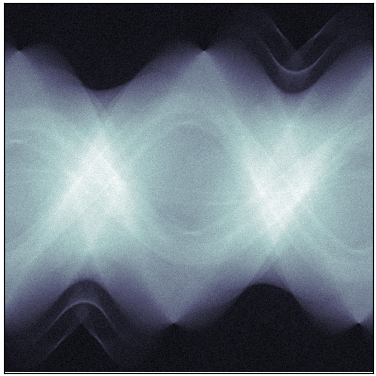}
  \end{subfigure}
  \begin{subfigure}[t]{.32\linewidth}
    \includegraphics[width=\linewidth]{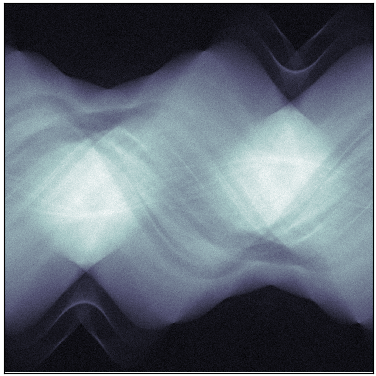}
  \end{subfigure}
  \begin{subfigure}[t]{.32\linewidth}
    \includegraphics[width=\linewidth]{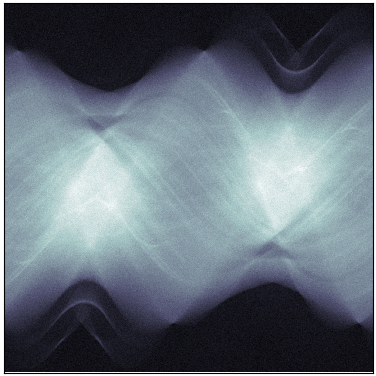}
  \end{subfigure}
  \caption{The top row shows three examples of phantoms used for generating data. These phantoms take values between $[0.0, 3.25]$, but all images are shown using a window set to $[0.8, 1.2]$ in order to enhance contrast of clinically more relevant details. The lower row shows corresponding simulated, noisy sinograms.
  }
  \label{fig:mayo_phantoms}
\end{figure}

For each algorithm, the number of unrolled iterations, corresponding to the
depth of the network, was set to $n_{\rm max} = 10$, and all evaluations have
been done with this depth.
However, in order to heuristically induce better stability of the general
schemes, we have trained using a stochastic depth as follows:
In each step of the training, the depth of the network has been set to the
outcome of the heavy-tailed random variable $\stochastic{n}_{\rm max} = \min \big[
\texttt{round}(8 + \stochastic{Z}), \, 100 \big]$,
where $\stochastic{Z}$ is the exponential of a Gaussian random variable with standard
deviation $1.25$ and mean value $\log\p{2} - 1.25^2/2$, so that
$\Expect[\stochastic{Z}] = 2$. The limitation to $100$ iterations is due
to limits in computational resources.

In order to improve stability and generalization properties of the trained
networks, we have normalized the operators before training, i.e., rescaled them
so that $\norm{\ForwardOp}_2 = \norm{\nabla}_2 = 1$. For the same reasons, we
have used the zero vector as initial guess for all networks.
Training has been done using the \emph{Adam} solver \cite{kingma2014adam}, with
standard parameter values except for $\beta_2 = 0.99$. Moreover, we have used
gradient clipping to limit the norm of the gradient of the training cost
function \eqref{eq:training_cost_function} to be less than or equal to one
\cite{pascanu2012understanding}.
As step length (learning rate) we have used a cosine annealing scheme
\cite{loshchilov2016sgdr}, i.e., a step length which in step $t$ takes the value
\[
  \eta_t = \frac{\eta_0}{2}\p{1 + \cos \p{\pi \frac{t}{t_{\rm max}}}},
\]
where the initial step length $\eta_0$ has been set to $10^{-3}$. We have
trained for $t_{\rm max} = 100 \, 000$ steps and have used 9 out of 10 phantoms
from the AAPM Low Dose CT Grand Challenge for training and one for evaluation.

All algorithms have been implemented using ODL \cite{adler2017ODL}, the GPU
accelerated version of ASTRA \cite{palenstijn2011performance, van2016fast}, and
Tensorflow \cite{abadi2016tensorflow}. The source code to replicate the experiments
is available online, where the weights of the trained networks are also explicitly given.\footnote{\url{https://github.com/aringh/data-driven_nonsmooth_optimization}}
We have used this setup to train
the following methods.

\begin{description}
\item[PDHG method.]
This corresponds to optimal selection of the parameters $\theta$, $\tau$, and
$\sigma$ for the PDHG method \eqref{eq:alg_CP} on the family of cost functions
\eqref{eq:CT_cost_function}.
In order to achieve this, we need to enforce the constraints $\theta \in [0,1]$
and $\sigma \tau \norm{L}^2 < 1$.
This has been done implicitly by a change of variables, namely by 
\begin{equation}
\theta = \frac{\ee^{s_1}}{1 + \ee^{s_1}}, \qquad \tau = \frac{1}{\norm{L}} \cdot
\frac{\ee^{s_2 + s_3}}{1 + \ee^{s_2}}, \qquad \sigma = \frac{1}{\norm{L}} \cdot \frac{\ee^{s_2 - s_3}}{1 + \ee^{s_2}} \label{eq:change_of_variables}
\end{equation}
with $s_1, s_2, s_3 \in \Real$. Here, $s_2$ determines how close the parameters $\sigma$ and $\tau$ are to the constraint $\sigma \tau \norm{L}^2 < 1$, while $s_3$ determines the trade-off between $\tau$ and $\sigma$. 

\item[PDHG method without constraints on the parameters.]
Here we train the same parameters $\theta, \tau, \sigma$ as in the PDHG method.
However, we do not make the change of variables \eqref{eq:change_of_variables},
therefore, no constraints on $\theta$, $\tau$, and $\sigma$ are enforced in the
training. This means that the resulting scheme might not correspond to a
globally convergent optimization algorithm.

\item[Proposed method from Section~{\ref{sec:new_solver}}.]
This corresponds to optimal parameter selection for the method
\eqref{eq:opt_alg} on the family of cost functions \eqref{eq:CT_cost_function}.
To adhere to the constraints in the assumptions in
Theorem~\ref{thm:convergence}, we have used the same kind of variable change as
in \eqref{eq:change_of_variables}, namely
\[
  \sndparam = \frac{2\ee^{s_1}}{1 + \ee^{s_1}}, \quad
  \fstparam = \frac{2\ee^{s_2}}{1 + \ee^{s_2}}, \quad
  \sigma = \frac{K}{\norm{L}} \cdot \frac{\ee^{s_3 - s_4}}{1 + \ee^{s_3}}, \quad
  \tau = \frac{K}{\norm{L}} \cdot \frac{\ee^{s_3 + s_4}}{1 + \ee^{s_3}}
\]
with $s_1, \ldots, s_4 \in \Real$, where $K = \frac{\sndparam^2 \p{2 - \sndparam} \p{2 - \fstparam}}{\p{\sndparam + \fstparam -
    \sndparam \fstparam}^2}$, as in \eqref{eq:sigma_tau_bound}.

\item[Parametrization proposed in Section~
  {\ref{subsec:param_learned_opt_solver}}.]
  Here, we have trained schemes of the form \eqref{eq:really_general_scheme}. We have
  done this for constant sequences of matrices $\SchemeMatrixC_1$, $\SchemeMatrixC_2$,
  $\SchemeMatrixB_1$, $\SchemeMatrixB_2$, $\SchemeMatrixA$, and $\SchemeMatrixD$. We
  restricted ourselves to the sizes $N = M = 2$ and $N = M = 3$.

\end{description}
\subsection{Performance of the trained methods}
To obtain an estimation of the true optimal value of
\eqref{eq:CT_cost_function}, we have run 1\,000 iterations of PDHG with
parameters as in \cite{sidky2012convex}.
In Table~\ref{tab:eval_CT} we show the difference between the obtained objective function value and
the minimal objective function value, averaged over $100$ samples.
As can be seen, the scheme proposed in Section~\ref{sec:learning_opt_solver} with $N=M=3$
performs best at 10 iterations.
Moreover, a general trend seems to be that more parameters in the 
algorithms improve the performance.
Finally, the results from one specific phantom are presented as reconstructions
in Figure~\ref{fig:results_reco}. Note that the reconstruction by PDHG with
parameters as in \cite{sidky2012convex} is left out due to the page layout.

\begin{table}
  \caption{Loss function values for the CT reconstruction after 10 iterations.
  The values given are of the form $\tfrac{1}{100} \sum_{i = 1}^{100} H_{b_i}(x_{10}) - H_{b_i}(x_{i}^*)$, i.e., the difference of the obtained objective function value and an estimate of the true minimum objective function value $H_{b_i}(x_{i}^*)$ corresponding to data $b_i$, averaged over 100 samples.}
  \label{tab:eval_CT}
  \begin{center}
    \begin{tabular}{l c}
      \toprule
      Method & Loss function values \\
      \midrule
      PDHG with parameters from \cite{sidky2012convex}                  & $109.93$ \\
      Trained PDHG with constraints on parameters                       & $82.381$ \\
      Trained solver \eqref{eq:opt_alg}                               & $24.183$ \\
      Trained PDHG without constraints on parameters                    & $27.761$ \\
      Trained scheme of type \eqref{eq:really_general_scheme} with $N = M = 2$ & $20.024$ \\
      Trained scheme of type \eqref{eq:really_general_scheme} with $N = M = 3$ & $\bm{14.905}$ \\
      \bottomrule
    \end{tabular}
  \end{center}
\end{table}

\begin{figure}
  \centering	
  \begin{subfigure}[t]{.32\linewidth}
    \includegraphics[width=\linewidth]{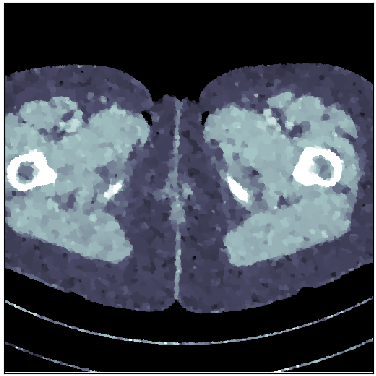}
    \caption{TV reconstruction.}\label{fig:tv_reco}
  \end{subfigure}
  \begin{subfigure}[t]{.32\linewidth}
    \includegraphics[width=\linewidth]{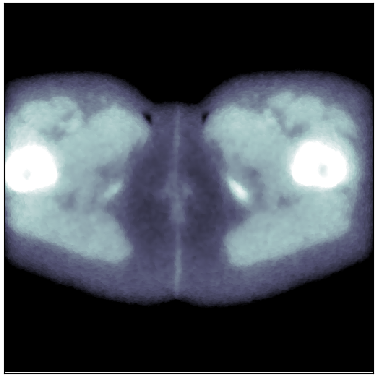}
    \caption{Trained PDHG with constraints on parameters.}
  \end{subfigure}
  \begin{subfigure}[t]{.32\linewidth}
    \includegraphics[width=\linewidth]{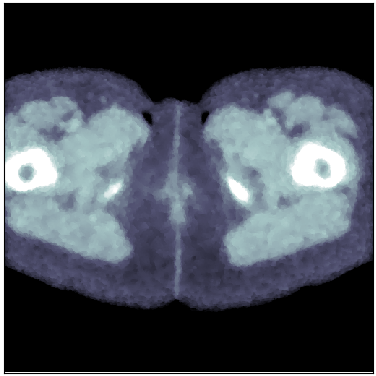}
    \caption{Trained solver \eqref{eq:opt_alg}.}
  \end{subfigure}
  \\[1em]
  \begin{subfigure}[t]{.32\linewidth}
    \includegraphics[width=\linewidth]{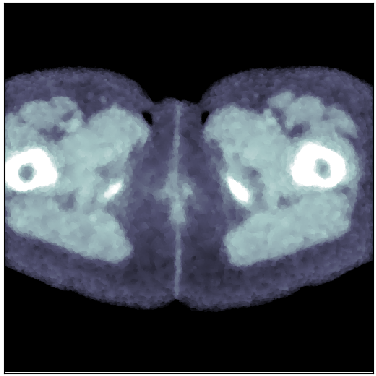}
    \caption{Trained PDHG without constraints on parameters.}
  \end{subfigure}
  \begin{subfigure}[t]{.32\linewidth}
    \includegraphics[width=\linewidth]{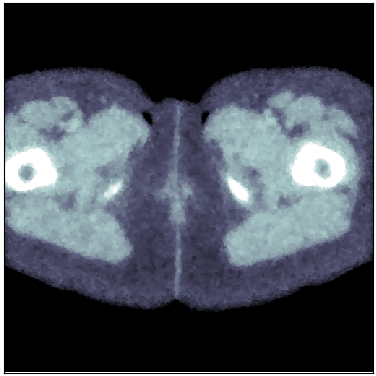}
    \caption{Trained scheme of type \eqref{eq:really_general_scheme} with $N = M = 2$.}
  \end{subfigure}
  \begin{subfigure}[t]{.32\linewidth}
    \includegraphics[width=\linewidth]{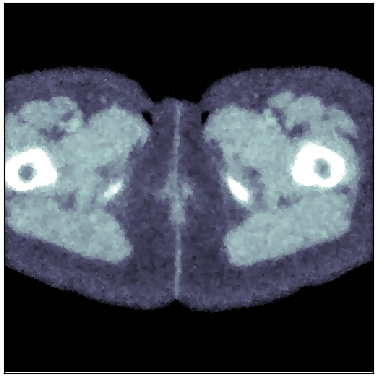}
    \caption{Trained scheme of type \eqref{eq:really_general_scheme} with $N = M = 3$.}
  \end{subfigure}
  \caption{Reconstruction with data from a phantom that was not used in the training. The TV reconstruction, to which they should be compared, is shown in \subref{fig:tv_reco}. All reconstructions use 10 steps.
   The phantom takes values between $[0.0, 2.33]$, but all images are shown using a window set to $[0.8, 1.2]$ in order to enhance contrast of clinically more relevant details.
  }
  \label{fig:results_reco}
\end{figure}

\subsubsection{Generalization to other iteration numbers}
Figure~\ref{fig:obj_fun} shows the objective function value
\eqref{eq:CT_cost_function} as a function of the iteration number, i.e., how
well the learned algorithms generalize to iteration numbers they are not trained
for.
For the trained, convergent solvers, the objective function value keeps
decreasing as expected. Furthermore, the solver
proposed in \eqref{eq:opt_alg} performs better than the others
also when the number of iterations are increased, but poorer in the beginning.
For the other schemes, it can be noted that, while training more parameters 
seems to increase the performance after $10$ iterations, it also seems
to decrease the  generalizability of the algorithm with respect to an
increase in the number of iterations.

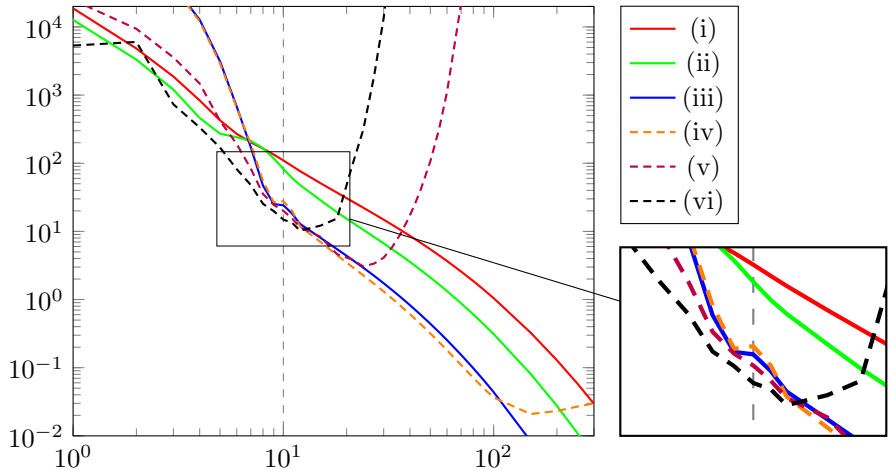
\begin{figure}
  \centering
  \begin{tikzpicture}[spy using outlines={rectangle, magnification=2.0, connect spies}]
    \coordinate (legend) at (7.2, 1.0);
    \coordinate (spyviewer) at (7.2, 2.5);
    \begin{loglogaxis}
      [
      ymin=.01,
      ymax=20000,
      xmin=1,
      xmax=300,
      legend style = {anchor=south west, at={(legend)}},
      no markers,
      cycle list={{red, solid, thick},
                  {green, solid, thick},
                  {blue, solid, thick},
                  {orange, densely dashed, thick},
                  {purple, densely dashed, thick},
                  {black, densely dashed, thick}}
      ]
      \draw[gray, dashed] (10, .0001) -- (10, 40000);
      \addplot table {plot_data/sidky.tab};
      \addplot table {plot_data/cp_const.tab};
      \addplot table {plot_data/our_2x2.tab};
      \addplot table {plot_data/cp.tab};
      \addplot table {plot_data/general_nm2.tab};
      \addplot table {plot_data/general_nm3.tab};
      \legend{(i), (ii), (iii), (iv), (v), (vi)}

      \coordinate (spypoint) at (axis cs:10,30);
      \spy[width=3.5cm,height=2.5cm] on (spypoint) in node [fill=white, anchor=north west] at (spyviewer);
    \end{loglogaxis}
  \end{tikzpicture}
  \caption{The figure shows the values $\tfrac{1}{100} \sum_{i = 1}^{100} H_{b_i}(x_{n}) - H_{b_i}(x_{i}^*)$, where $H_{b_i}(x_{i}^*)$ is an estimate of the true minimum objective function value corresponding to data $b_i$, of several
    reconstruction methods as a function of the iteration number $n$.
    Solid lines are real optimization solvers, dotted lines are schemes that
    might not converge to the true optimal solution.
    (i) PDHG with parameters as in \cite{sidky2012convex},
    (ii) PDHG with trained parameters with constraints,
    (iii) proposed solver \eqref{eq:opt_alg} with trained parameters,
    (iv) PDHG with trained free parameters,
    (v) proposed scheme \eqref{eq:really_general_scheme} with $N=M=2$, and
    (vi) proposed scheme \eqref{eq:really_general_scheme} with $N=M=3$.}%
  \label{fig:obj_fun}%
\end{figure}

\subsubsection{Generalization to deblurring}\label{sec:deblurring}
Next, we investigate the generalizability of the trained networks to other
optimization problems by replacing the forward operator $\ForwardOp$ in
\eqref{eq:CT_cost_function} with a convolution. This corresponds to another TV
problem in imaging, namely image deblurring.

Clearly, the trained networks that correspond to optimization solvers 
with convergence guarantees can be applied to other convex optimization 
problems. (Note that we still normalize the
operators to have operator norm one so that the assumptions in Theorem
\ref{thm:convergence} do not change.)
However, nothing guarantees that parameters that give fast convergence 
on one type of problems will also give fast convergence on another one.

Two example images are shown in Figure~\ref{fig:deconv_ground_truths}. The
images in Figure \ref{subfig:Raccoon_original}--\ref{subfig:Raccoon_TV}, corresponding to
the \enquote{Raccoon} test image, are of size $1024\times 768$ and use a
different regularization parameter.
Blurring has been done with Gaussian kernels.
For the \enquote{Ascent} test image, the kernel has a standard deviation of
approximately three pixels in each direction, whereas for the \enquote{Raccoon}
test image, the kernel has a standard deviation of approximately four pixels in the up-down and six
pixels in the left-right direction.
As for the sinograms in the CT example, $5\%$ white noise has been added to
the blurred images.
Again, to obtain an estimation of the true optimal value of
 we have run 1\,000 iterations of PDHG with
parameters as in \cite{sidky2012convex}.
For each algorithm, the difference between the obtained objective function value and minimal objective function value is
presented in Table~\ref{tab:eval_deconv},
 and the deblurred images are shown in
Figures~\ref{fig:results_deblurring_ascent} and
\ref{fig:results_deblurring_raccoon}. Again, the reconstruction by PDHG with
parameters as in \cite{sidky2012convex} is left out due to the page layout.

The method with $N=M=3$ does not generalize well. However, the method with $N=M=2$ generalizes, and the optimization algorithm from Section~\ref{sec:new_solver}, with trained parameters, is one of the best on these two test problems.

\begin{figure}
  \centering	
  \begin{subfigure}[t]{.32\linewidth}
    \includegraphics[width=\linewidth]{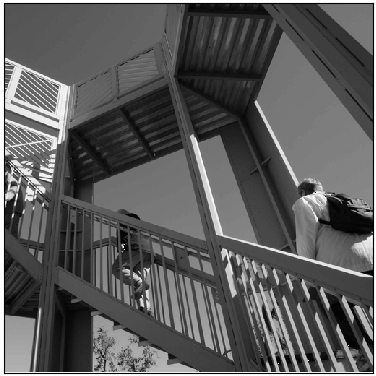}
    \caption{\enquote{Ascent} test image.}
  \end{subfigure}
  \begin{subfigure}[t]{.32\linewidth}
    \includegraphics[width=\linewidth]{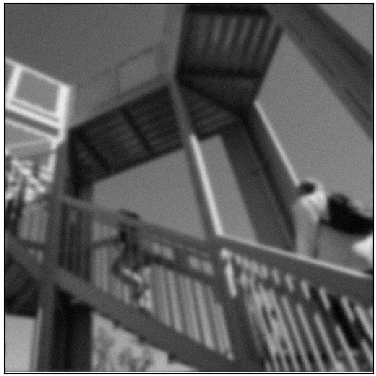}
    \caption{Blurred, noisy image.}
  \end{subfigure}
  \begin{subfigure}[t]{.32\linewidth}
    \includegraphics[width=\linewidth]{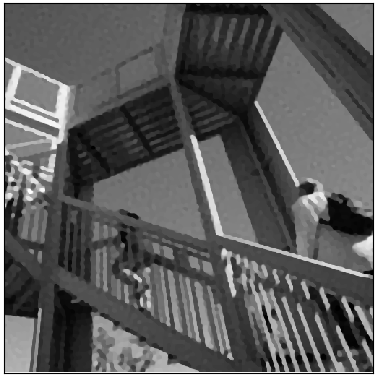}
    \caption{TV reconstruction.}
  \end{subfigure}
  \\[1em]
  \begin{subfigure}[t]{.32\linewidth}
    \includegraphics[width=\linewidth]{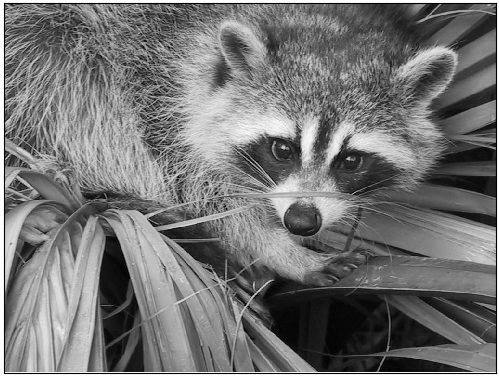}
    \caption{\enquote{Raccoon} test image.}
    \label{subfig:Raccoon_original}
  \end{subfigure}
  \begin{subfigure}[t]{.32\linewidth}
    \includegraphics[width=\linewidth]{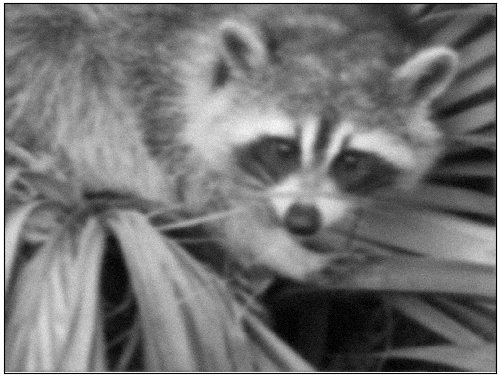}
    \caption{Blurred noisy image.}
    \label{subfig:Raccoon_blurry}
  \end{subfigure}
  \begin{subfigure}[t]{.32\linewidth}
    \includegraphics[width=\linewidth]{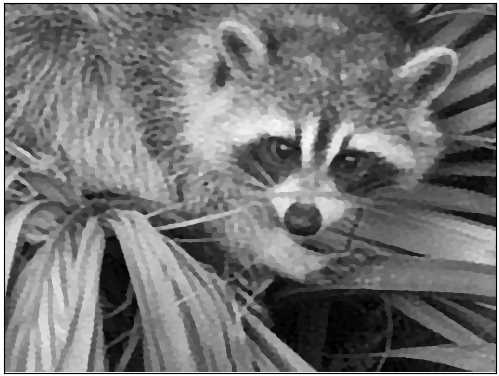}
    \caption{TV reconstruction.}
    \label{subfig:Raccoon_TV}
  \end{subfigure}
  \caption{Example images used for the deblurring problem in Section~\ref{sec:deblurring}.
  }
  \label{fig:deconv_ground_truths}
\end{figure}

\begin{table}
  \caption{Loss function values for the deblurring problem in
    Section~\ref{sec:deblurring}.
    Here, $H_{b_i}(x_{i}^*)$ is an estimate of the true minimum objective function value corresponding to data $b_i$.}
  \label{tab:eval_deconv}
  \begin{center}
    \begin{tabular}{l c c}
      \toprule
      Method & \multicolumn{2}{c}{$H_{b_i}(x_{10}) - H_{b_i}(x_{i}^*)$} \\
      \cmidrule(r){2-3}
                                                       & Ascent & Raccoon      \\
      \midrule
      PDHG with parameters from \cite{sidky2012convex}                  & $5.514$        & $11.475$ \\
      Trained PDHG with constraints on parameters                       & $4.256$        & $8.5126$ \\
      Trained solver \eqref{eq:opt_alg}                               & $\bm{2.173}$        & $4.5898$ \\
      Trained PDHG without constraints on parameters                    & $2.204$        & $\bm{4.4790}$ \\
      Trained scheme of type \eqref{eq:really_general_scheme} with $N = M = 2$ & $3.514$        & $9.9139$ \\
      Trained scheme of type \eqref{eq:really_general_scheme} with $N = M = 3$ & $208.37$       & $873.33$ \\
      \bottomrule
    \end{tabular}
  \end{center}
\end{table}

\begin{figure}
  \centering	
  \begin{subfigure}[t]{.32\linewidth}
    \includegraphics[width=\linewidth]{Ascent/sidky_iter_1000}
    \caption{TV reconstruction.}
  \end{subfigure}
  \begin{subfigure}[t]{.32\linewidth}
    \includegraphics[width=\linewidth]{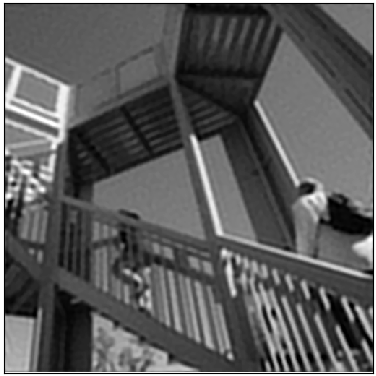}
    \caption{Trained PDHG with constraints on parameters.}
  \end{subfigure}
  \begin{subfigure}[t]{.32\linewidth}
    \includegraphics[width=\linewidth]{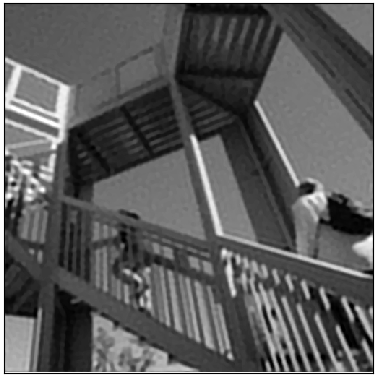}
    \caption{Trained solver \eqref{eq:opt_alg}.}
  \end{subfigure}
  \\[1em]
  \begin{subfigure}[t]{.32\linewidth}
    \includegraphics[width=\linewidth]{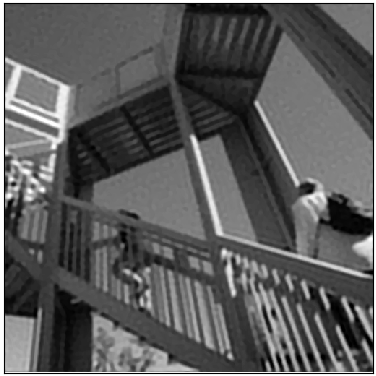}
    \caption{Trained PDHG without constraints on parameters.}
  \end{subfigure}
  \begin{subfigure}[t]{.32\linewidth}
    \includegraphics[width=\linewidth]{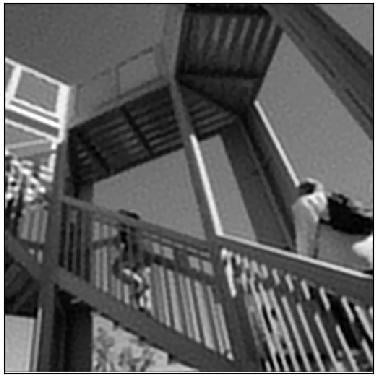}
    \caption{Trained scheme of type \eqref{eq:really_general_scheme} with $N = M = 2$.}
  \end{subfigure}
  \begin{subfigure}[t]{.32\linewidth}
    \includegraphics[width=\linewidth]{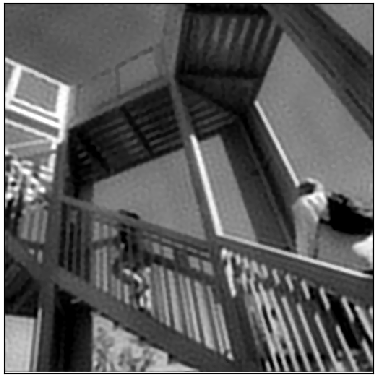}
    \caption{Trained scheme of type \eqref{eq:really_general_scheme} with $N = M = 3$.}
  \end{subfigure}
  \caption{Reconstructions with the trained algorithms for the \enquote{Ascent}
    image.
  }
  \label{fig:results_deblurring_ascent}
\end{figure}

\begin{figure}
  \centering	
  \begin{subfigure}[t]{.32\linewidth}
    \includegraphics[width=\linewidth]{Raccoon/sidky_iter_1000}
    \caption{TV reconstruction.}
  \end{subfigure}
  \begin{subfigure}[t]{.32\linewidth}
    \includegraphics[width=\linewidth]{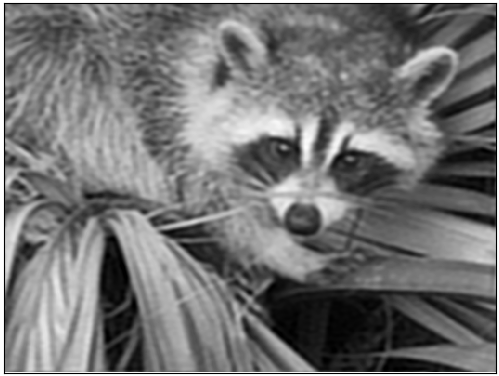}
    \caption{Trained PDHG with constraints on parameters.}
  \end{subfigure}
  \begin{subfigure}[t]{.32\linewidth}
    \includegraphics[width=\linewidth]{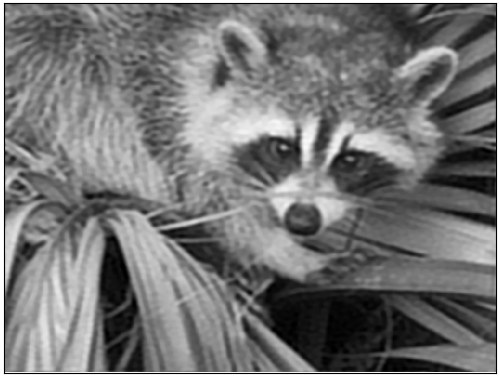}
    \caption{Trained solver \eqref{eq:opt_alg}.}
  \end{subfigure}
  \\[1em]
  \begin{subfigure}[t]{.32\linewidth}
    \includegraphics[width=\linewidth]{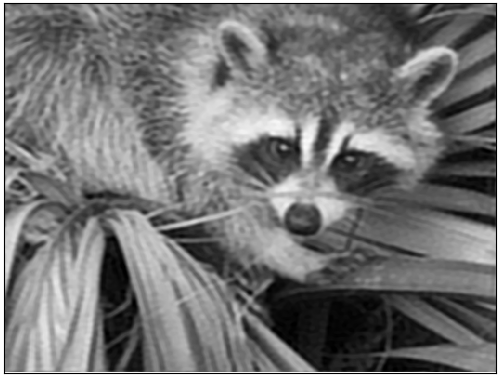}
    \caption{Trained PDHG without constraints on parameters.}
  \end{subfigure}
  \begin{subfigure}[t]{.32\linewidth}
    \includegraphics[width=\linewidth]{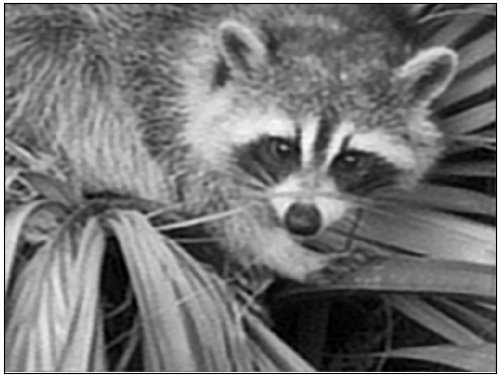}
    \caption{Trained scheme of type \eqref{eq:really_general_scheme} with $N = M = 2$.}
  \end{subfigure}
  \begin{subfigure}[t]{.32\linewidth}
    \includegraphics[width=\linewidth]{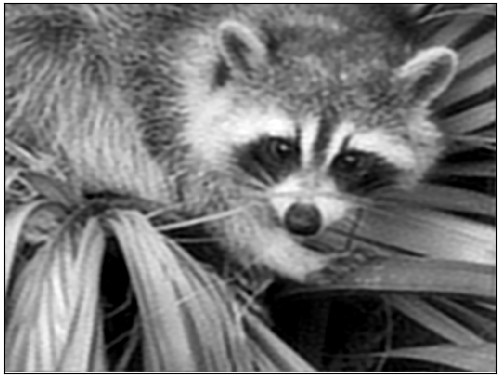}
    \caption{Trained scheme of type \eqref{eq:really_general_scheme} with $N = M = 3$.}
  \end{subfigure}
  \caption{Reconstructions with the trained algorithms for the \enquote{Raccoon}
    image.}
  \label{fig:results_deblurring_raccoon}
\end{figure}

\section{Conclusions and future work}
In this work, we have first proposed a new solver for maximally monotone
inclusion problems and proved convergence guarantees.
In particular, we have also proposed a new convergent primal-dual proximal
solver for convex optimization problems.
Further, we have investigated new aspects of learning an optimization solver.
This is particularly relevant in inverse problems where one can parametrize the
objective function by data, leaving the other parts unchanged.
This can, in fact, also be interpreted as learning a pseudo-inverse of the
forward operator in an unsupervised fashion.
Moreover, the framework admits enforcing convergence and stability properties in
the learning.
We should emphasize that this implies a form of generalizability to other data,
and even other forward operators, since the scheme cannot diverge.

There are several different directions in which the work from this article can
be extended: Regarding the optimization perspective, one could investigate
whether \eqref{eq:stability} can be further relaxed to introduce more free
parameters while retaining convergence, e.g. by relaxing \eqref{eq:a11andeq:c11}
or letting parameters vary in each iteration.

Also from a machine learning perspective, there are aspects to be further
investigated:
\begin{itemize}
\item Since accelerated first-order algorithms like FISTA \cite{beck2009afast}
  can be parametrized by \eqref{eq:really_general_scheme}, does the learning
  result in a scheme with $\Ordo\p{1/n^2}$ convergence rate for the objective
  function values when trained for $n$ iterations?
  
\item Our numerical experiments suggest that training without \enquote{convergence
  constraints} gives the network more freedom and thereby improves accuracy.
  However, the resulting schemes seem to be unstable beyond the fixed number of
  iterates used for training.
  Is it true that, in general, convergence cannot be enforced by training alone?

\item Is it possible to state and prove a time accuracy trade-off theorem, i.e.,
  to estimate the error between the trained solver and the true solution to the
  optimization?
  If so, which properties of the underlying family of objective functions
  (training data) does this require?
\end{itemize}

\section*{Acknowledgments}
We acknowledge Swedish Foundation of Strategic Research grants AM13-0049 and
ID14-0055, Swedish Research Council grant 2014-5870 and support from
Elekta.

The authors thank Dr.\ Cynthia McCollough, the Mayo Clinic, and the American
Association of Physicists in Medicine for providing the data necessary for
performing comparison using a human phantom.

\printbibliography
\end{document}